\documentclass[11pt]{article}
\usepackage{amsmath,amsfonts,amssymb, amsthm,enumerate,graphicx}
\usepackage{tikz,authblk}
\usepackage{subfig}
\usepackage{url}

\newcommand{\lk}[2]{{\rm lk}_{#1}(#2)}
\newcommand{\st}[2]{{\rm st}_{#1}(#2)}

\newtheorem{Lemma}{Lemma}[section]
\newtheorem{Theorem}[Lemma]{Theorem}
\newtheorem{Proposition}[Lemma]{Proposition}
\newtheorem{Corollary}[Lemma]{Corollary}
\newtheorem{Remark}[Lemma]{Remark}
\newtheorem{definition}[Lemma]{Definition}

\def\vst{\vskip .1cm}

\def\t{\tau}
\def\a{\alpha}
\def\s{\sigma}
\def\p{\partial}

\def\H{{\mathcal{H}}}
\def\R{{\mathbb{R}}}

\def\D{\Delta}
\def\lk{lk_{\Delta}}
\def\st{st_{\Delta}}

\begin{document}
\title{A structure theorem for homology 4-manifolds with $g_2\leq 5$}
\author{Biplab Basak and Sourav Sarkar}
\date{}
\maketitle
\vspace{-10mm}
\noindent{\small Department of Mathematics, Indian Institute of Technology Delhi, New Delhi 110016, India.}

\footnotetext[1]{{\em E-mail addresses:} \url{biplab@iitd.ac.in} (B.
Basak), \url{Sourav.Sarkar@maths.iitd.ac.in} (S. Sarkar).}
\begin{center}
\date{April 18, 2024}
\end{center}
\hrule

\begin{abstract}
Numerous structural findings of homology manifolds have been derived in various ways in relation to $g_2$-values. The homology $4$-manifolds with $g_2\leq 5$ are characterized combinatorially in this article. It is well-known that all homology $4$-manifolds for $g_2\leq 2$ are polytopal spheres. We demonstrate that homology $4$-manifolds with $g_2\leq 5$ are triangulated spheres and are derived from triangulated 4-spheres with $g_2\leq 2$ by a series of connected sum, bistellar 1- and 2-moves, edge contraction, edge expansion, and edge flipping operations. We establish that the above inequality is optimally attainable, i.e., it cannot be extended to $g_2 = 6$.
\end{abstract}

\noindent {\small {\em MSC 2020\,:} Primary 05E45; Secondary 05C30, 57Q15, 57Q25.

}

\medskip
\section{Introduction}
The $g$-conjecture \cite{McMullen1971} proposed a detailed characterization of $f$-vectors of simplicial polytopes, as well as a search for potential $f$-vectors of other types of simplicial complexes. The components of the $g$-vector of any simplicial $d$-polytope are shown to be non-negative in \cite{Stanley}. The sufficiency for $f$-vectors of simplicial polytopes and simplicial convex polytopes proposed in the $g$-conjecture can be found in \cite{Billera1980, Billera1981}. This naturally raises the issue of the general complexes' geometric arrangement in relation to the $g$-values. It is trivial that $g_1(\D)=0$ holds if and only if $\D$ is a simplex.

Several classification results have been derived based on the third component $g_2$ of the $g$-vector. The well-known Lower Bound Theorem (LBT), established by Barnette \cite{Barnette1, Barnette2}, states that if $\D$ is the boundary complex of a simplicial $(d+1)$-polytope or, more broadly, a finite triangulation of a connected compact $d$-manifold without boundary, where $d\geq 2$, then $g_2(\D)\geq 0$. Every triangulation of a 3-manifold with $g_2\leq 9$ is a triangulated 3-sphere, and it is stacked when $g_2=0$ \cite{Walkup}. Moreover, the inequality $g_2\leq 9$ for triangulated 3-spheres is optimal, as there are triangulations of 3-dimensional handles (both orientable and non-orientable) with $g_2=10$. The inequality $g_2(\D)\geq 0$ has been established in \cite{BagchiDatta, Fogelsanger, Kalai} for the class of all normal pseudomanifolds of dimensions at least three. As a step toward equality on $g_2$, for $d\geq 3$, the stacked $d$-spheres have been shown to be the only normal $d$-pseudomanifolds with $g_2=0$ (cf. \cite{Kalai, Tay}). Some structural results on normal pseudomanifolds can be found in \cite{BagchiDatta98, BasakSwartz}.

If a homology $d$-manifold has no missing facets, then it is referred to as {\em prime}. Combinatorial characterizations of prime homology $d$-manifolds with $g_2\leq2$ are known due to Nevo and Novinsky \cite{NevoNovinsky} for $g_2=1$ and due to Zheng \cite{Zheng} for $g_2=2$.
\begin{Proposition}{\rm \cite{NevoNovinsky}}\label{Nevod>3}
Suppose $d\geq 3$, and let $\D$ be a prime homology $d$-sphere with $g_2(\D)=1$. Then $\D$ is combinatorially isomorphic to either the join of the boundary complex of two simplices, where each simplex has a dimension of at least $2$ and their dimensions add up to $d+1$, or the join of a cycle and the boundary complex of a $(d-1)$-simplex.
\end{Proposition} 
\begin{Proposition}{\rm\cite{Zheng}}\label{Zhengd>3}
Suppose $d\geq 4$, and let $\D$ be a prime homology $d$-manifold with $g_2(\D)=2$. Then $\D$ can be obtained by a central retriangulation of a polytopal $d$-sphere with $g_2=0$ or $1$, along some stacked subcomplex.
\end{Proposition}  

A homology manifold with $g_2\leq 2$ is the connected sum of a finite number of prime homology manifolds with $g_2$ less than or equal to 2. Therefore, a homology $d$-manifold with $g_2\leq 2$ is a triangulated $d$-sphere, as shown in the following result.
 \begin{Proposition}{\rm\cite{NovikSwartz2020}}\label{NovikSwartz2020}
 Let $d\geq 3$, and let $\D$ be a homology $d$-manifold that is not prime.  Then $\D$ is either a connected sum of prime homology manifolds or the result of a handle addition on a homology manifold.
 \end{Proposition}

In this article, our goal is to establish an upper bound, say $b$, on $g_2$ such that if $\D$ is a homology 4-manifold with $g_2(\D)\leq b$, then $\D$ is a triangulated 4-sphere. Furthermore, we will outline a combinatorial structure for such triangulated spheres. Our approach is rooted in the rigidity theory of graphs and fundamental combinatorial operations on simplicial complexes.

\begin{Theorem}\label{main}
If $\D$ is a homology $4$-manifold with $g_2(\D)\leq 5$, then $\D$ is a triangulated sphere obtained from a triangulated $4$-sphere with $g_2\leq 2$ by a sequence of the following operations:
\begin{enumerate}[$(i)$]
\item a connected sum,
\item a bistellar $1$-move and an edge contraction,
\item an edge expansion together with flipping an edge (possibly zero times) or a bistellar $2$-move (possibly zero times).
\end{enumerate}
Moreover, the above inequality is the best possible, i.e., it cannot be extended to $g_2 = 6$.
\end{Theorem}

\begin{Remark}\label{sharp}
{\rm In \cite{BrehmKuhnel}, Brehm and K\"{u}hnel proved that the graph of the 9-vertex combinatorial triangulation of complex projective space $\mathbb{CP}^2_9$ is complete. Consequently, $g_2(\mathbb{CP}^2_9)=6$. Therefore, the upper bound for $g_2(\D)$, for which the homology $4$-manifold $\D$ is a triangulated $4$-sphere, is the most favorable outcome, i.e., it cannot be expanded to accommodate $g_2 = 6$.


Note that $g_3(\mathbb{CP}^2_9)=10$, and $g_2(lk_{\mathbb{CP}^2_9} v)=6$ for every vertex $v$ in $\mathbb{CP}^2_9$. Moreover, the triangulation $\mathbb{CP}^2_9$ is a prime simplicial complex with all prime vertex-links. However, our constructions depend on the concept that for prime simplicial complexes $\D$ with all prime vertex-links, whenever $ g_2(\D)\leq 5$, we have $g_3(\D)=0$, and there is a vertex $u$ in $\D$ such that $g_2(\lk u)\leq 2$. Therefore, our constructions are not applicable for $\mathbb{CP}^2_9$.}
\end{Remark}

\section{Preliminaries}
All the simplices considered in this article are geometric, and the simplicial complexes are finite. We also assume that every simplicial complex contains $\emptyset$ as the only simplex of dimension $-1$. If $\sigma$ is an $n$-simplex in $\mathbb{R}^m$ for some $m$, which is the convex hull of $n+1$ affinely independent points $v_0,v_1,\dots, v_n$, then we write the simplex $\sigma$ as $v_0v_1\dots v_n$. If $\D$ is a simplicial complex, then the {\em geometric carrier} of $\D$ is the union of all simplices in $\D$, together with the subspace topology induced from $\mathbb{R}^m$ for some $m$. We say a simplicial complex $\Delta$ is {\em pure} if all its maximal simplices are of the same dimension. Every simplex present in a simplicial complex $\Delta$ is called a {\em face} of $\Delta$, and the maximal faces of $\Delta$ are called {\em facets}. The set of all vertices (or $0$-simplices) of $\Delta$ is denoted by $V(\Delta)$, the set of all edges (or $1$-simplices) of $\Delta$ is denoted by $E(\Delta)$, and the collection of all vertices and edges of $\Delta$ is called the {\em graph} of $\Delta$ and is denoted by $G(\Delta)$. For a subset $V'\subseteq V(\Delta)$, we define $\Delta[V']:=\{\sigma\in\Delta :V(\sigma)\subseteq V'\}$. A simplex $\sigma$ is called a {\em missing face} (or {\em missing simplex}) of $\Delta$ if $\partial(\sigma)\subseteq\Delta$ but $\sigma\notin\Delta$. A pure simplicial complex is called {\em prime} if it has no missing facets.

Two simplices $\sigma = u_0u_1\cdots u_k$ and $\tau = v_0v_1\cdots v_l$ in $\mathbb{R}^m$ for some $m\in \mathbb{N}$, are {\em independent} if $u_0,\dots ,u_k,v_0,\dots,v_l$ are affinely independent. In that case, $u_0\cdots u_kv_0\cdots v_l$ is a $(k + l + 1)$-simplex and is denoted by  $\sigma\star\tau$ or $\sigma\tau$. Two simplicial complexes $\D_1$ and $\D_2$ are said to be {\em independent} if $\s\t$ is a simplex of dimension $(i+j+1)$ for every pair of simplices $\s\in\D_1$ and $\t\in\D_2$ of dimensions $i$ and $j$, respectively. The join of two independent simplicial complexes $\D_1$ and $\D_2$ is the simplicial complex $\D_1\star\D_2:=\{\s\t : \s\in\D_1,\hspace{.15cm}\text{and}\hspace{.15cm}\t\in\D_2\}$. 
For a pair $(\sigma,\Delta)$, where $\sigma$ is a simplex and $\Delta$ is a simplicial complex, $\sigma\star\Delta$ denotes the simplicial complex $\{\alpha:\alpha\leq\sigma\}\star\Delta$. The link of a face $\sigma$ in $\Delta$ is the set $\{\gamma\in \Delta : \gamma\cap\sigma=\emptyset$ and $ \gamma\sigma\in \Delta\}$, denoted by $\lk \sigma$. Similarly, the star of a face $\sigma$ in $\Delta$ is the set $\{\alpha\in \Delta : \alpha\leq\sigma \beta$ and $\beta\in \lk \sigma\}$, denoted by $\st \sigma$. The number of vertices in $\lk u$ is called the {\em degree} of $u$, and is denoted by $d(u,\D)$ or simply $d(u)$. We use the notation $C(a_1,\dots,a_m)$ to denote a cycle of length $m$ (or an $m$-cycle) with vertices $a_1,\dots,a_m$. If the vertex set is not specified, we use the notation $C$ to denote a cycle.

One of the most common enumerative tools of a $d$-dimensional simplicial complex $\D$ is its $f$-vector $(f_{-1}(\D),f_0(\D),\dots,f_d(\D))$, where $f_i(\D)$ is the number of $i$-dimensional simplices present in $\D$. We also define the $h$-vector of $\D$ as $(h_0(\D),h_1(\D),\dots,h_d(\D))$, where
$$ h_i(\D)=\sum_{j=0}^{i}(-1)^{i-j} \binom{d+1-j}{i-j}f_{j-1}(\D),$$
 and we define  $g_i(\D):= h_i(\D)-h_{i-1}(\D)$. In particular, $g_2(\D)=f_1(\D)-(d+1)f_0(\D) + \binom{d+2}{2}$ and $g_3(\D)= f_2(\D)-df_1(\D)+ \binom{d+1}{2}f_0 - \binom{d+2}{3}$. In \cite{Swartz2004}, a relation between $g_i$'s of all pure simplicial complexes and their vertex-links is given, which was first stated by McMullen in \cite{McMullen} for shellable complexes.

 \begin{Lemma}{\rm \cite{Swartz2004}}\label{g-relations}
If $\D$ is a $d$-dimensional pure simplicial complex and $k$ is a positive integer, then
$$ \sum_{v\in\D} g_k(\lk v)=(k+1)g_{k+1}(\D) +(d+2-k)g_k(\D).$$
\end{Lemma}

Let $\D$ be a pure $d$-dimensional simplicial complex, and let $\mathbb{F}$ be a given field. We say that $\D$ is a $d$-dimensional {\em $\mathbb{F}$-homology sphere} (or a {\em homology $d$-sphere}) if the $i$-th reduced homology groups $\tilde{H}_{i}(\lk \sigma; \mathbb{F})\cong \tilde{H}_{i}(\mathbb{S}^{d-j-1}; \mathbb{F})$ for every face $\sigma\in\D$ (including the empty face) of dimension $j$. Similarly, $\D$ is a $d$-dimensional {\em $\mathbb{F}$-homology manifold} (or a {\em homology $d$-manifold}) if the link of every vertex in $\D$ is a $(d-1)$-dimensional $\mathbb{F}$-homology sphere. A strongly connected and pure simplicial complex $\D$ of dimension $d$ is called a {\em normal $d$-pseudomanifold} if 
every $(d-1)$-face of $\D$ is contained in exactly two facets in $\D$ and the links of all the simplices of dimension $\leq (d-2)$ are connected. The following result is easy to verify.

\begin{Lemma} \label{> 1 for prime u}
Suppose $d \geq 4$, and let $\D$ be a prime normal $d$-pseudomanifold such that $g_2(\D) \geq 1$. If $u$ is a vertex in $\D$ such that $\lk u$ is prime, then $g_2(\lk u) \geq 1$.
\end{Lemma}

Given two positive integers $a$ and $i$, there is a unique way to write
$$a={a_i\choose i}+{a_{i-1}\choose {i-1}}+\dots+{a_{j}\choose {j}},$$
where $a_i>a_{i-1}>\dots>a_j\geq j\geq 1$. Then $a^{<i>}$, the $i$-th {\em Macaulay pseudo-power} of $a$, is defined as 
$$a^{<i>}:= {a_i+1\choose i+1}+{a_{i-1}+1\choose {i}}+\dots+{a_{j}+1\choose {j+1}}.$$
\begin{Lemma}{\rm \cite{Swartz2014}}\label{g3 bound}
Let $\D$ be a normal $d$-pseudomanifold, where $d\geq 3$. Then $g_3(\D)\leq g_2(\D)^{<2>}$.
\end{Lemma}

\begin{Lemma}\label{g3=0}
Let $\D$ be a homology $4$-manifold with $1\leq g_2(\D)\leq 5$. Then $g_3(\D)=0$.
\end{Lemma}

  \begin{proof}
  Recall that $g_3(\D)=h_3(\D)-h_2(\D)$. Since $\D$ is a homology 4-manifold, it follows from \cite{Klee, NovikSwartz2009} that $h_{5-i}(\D)-h_i(\D)=(-1)^{i}{5\choose i}(\chi(\D)-\chi(\mathbb{S}^4))$. In particular, for $i=2$, $h_3(\D)-h_2(\D)=10(\chi(\D)-\chi(\mathbb{S}^4))$, i.e., $g_3(\D)$ is an integer multiple of 10. Since $g_2(\D)\leq 5$, it follows from Lemma \ref{g3 bound} that $g_3(\D)\leq 7$. Therefore, by combining these two we have $g_3(\D)\leq 0$.
  
On the other hand, Lemma \ref{g-relations} implies that
\begin{eqnarray}\label{g2 sum equation}
\sum_{u\in\D} g_2(\lk  u)&=& 3g_3(\D)+4g_2(\D).
\end{eqnarray} 
If $g_3(\D)\leq -10$, then it follows from Equation \eqref{g2 sum equation} that $\sum_{u\in\D} g_2(\lk  u)\leq -30+20=-10$, which is not possible. Therefore, $g_3(\D)=0$.
\end{proof}

 As mentioned earlier, the combinatorial description of homology manifolds with $g_2\leq 2$ is known \cite{NevoNovinsky, Zheng}. In particular,  homology $3$-manifolds with $g_2\leq 2$ have the following combinatorial structures.
 
 \begin{Proposition}\label{Nevo g2=1}{\rm\cite{NevoNovinsky}}
  Let $\D$ be a prime homology $3$-manifold with $g_2(\D)=1$. Then $\D$ is combinatorially isomorphic to $C\star\p(\s^2)$, where $C$ is a cycle of length $n$ and $\s^2$ is a $2$-simplex. 
 \end{Proposition}  
 \begin{Proposition}\label{Zheng g2=2}{\rm\cite{Zheng}}
  Let $\D$ be a prime homology $3$-manifold with $g_2(\D)=2$. Then $\D$ is one of the following:
  \begin{enumerate}[$(i)$]
  \item $\D$ is obtained from a triangulated $3$-sphere with $g_2=1$ by stellar subdivision at a ridge,
  \item  $\D$ is an octahedral $3$-sphere, i.e., $\D=\p(a_1b_1)\star\p(a_2b_2)\star\p(a_3b_3)\star\p(a_4b_4)$.
  \end{enumerate} 
  \end{Proposition}
  \subsection{Rigidity of graphs and Pseudomanifolds}
  Let $G$ be a graph, and let $V$ be the set of vertices of $G$. A {\em $d$-embedding} of $G$ into $\mathbb{R}^d$ is an injective map $f: V\to \mathbb{R}^d$. A {\em framework} of $G$ is a pair $(G,f)$, where $f$ is a $d$-embedding of $G$ for some $d\geq 1$. In this context, we always use the Euclidean norm. 

\begin{definition}
{\rm
 A $d$-embedding $f$ of a graph $G=(V,E)$ is called {\em rigid} if there is an $\epsilon >0$ such that for every $d$-embedding $g$ of $G$ satisfying $(i)\hspace{.25cm} \lVert f(v)-g(v)\rVert<\epsilon$ for every vertex $v\in V$, and $(ii)\hspace{.25cm}\lVert f(v)-f(u)\rVert=\lVert g(v)-g(u)\rVert$ for every edge $vu\in E$, one has $\lVert f(v)-f(u)\rVert=\lVert g(v)-g(u)\rVert$ for every pair of vertices $v,u$ in $V$.}
 \end{definition}
 An embedding $f$ that is not rigid is called \textit{flexible}. If a $d$-embedding is rigid, then the corresponding framework is called a {\em rigid framework}. We denote the set of all $d$-embeddings of a graph $G$ by $\mathcal{E}(G)$, and the set of all rigid $d$-embeddings of a graph $G$ by $R\mathcal{E}(G)$. Note that $\mathcal{E}(G)$ is a topological vector space of dimension $d\cdot |V|$ and $R\mathcal{E}(G)$ is a subspace of $\mathcal{E}(G)$. 
\begin{definition}
{\rm
A graph $G$ is called {\em generically $d$-rigid} if the set of all rigid $d$-embeddings $R\mathcal{E}(G)$ is an open and dense subset of the set of all $d$-embeddings $\mathcal{E}(G)$.
}
\end{definition}
\begin{Proposition}{\rm\cite[Gluing Lemma]{AshimowRoth,NevoNovinsky}}\label{Gluing Lemma}
Let $G_1$ and $G_2$ be two generically $d$-rigid graphs such that $G_1\cap G_2$ contains at least $d$ vertices. Then $G_1\cup G_2$ is a generically $d$-rigid.
\end{Proposition}

\begin{Lemma}\label{rigid embedding exists}
Let $G$ be a generically $d$-rigid graph, and let $H_1$ and $H_2$ be two generically $d$-rigid subgraphs of $G$. Then there is a rigid $d$-embedding $\Psi: V(G)\to \mathbb{R}^d$ whose restrictions on $V(H_1)$ and $V(H_2)$ are rigid $d$-embeddings of $H_1$ and $H_2$, respectively.
\end{Lemma}
\begin{proof}
Let $V(G)=\{v_1,\dots,v_n\}$ be the set of vertices of $G$. Since $G$ is a generically $d$-rigid, the set of all rigid $d$-embeddings $R\mathcal{E}(G)$ is an open and dense subset of the set of all $d$-embeddings $\mathcal{E}(G)$. Let $f\in R\mathcal{E}(G)$. Note that  $\mathcal{E}(G)$ is isomorphic to the Euclidean space $\mathbb{R}^{d\cdot |V|}$. Since $R\mathcal{E}(G)$ is an open subset of $\mathcal{E}(G)$, there is an open set  $U=U_1\times\cdots\times U_{|V|}$, where each $U_i\subseteq\mathbb{R}^d$ (with $i$ corresponds to the vertex $v_i$), such that $f\in U\subseteq R\mathcal{E}(G)\subseteq \mathcal{E}(G)$. Let the number of vertices in $H_1$ and $H_2$ be $n_1$ and $n_2$, respectively. Without loss of generality, let $U_1,\dots,U_{n_1}$ correspond to the vertices of $H_1$. Let $f_{H_{1}}$ be the restriction of $f$ on $V(H_1)$. Then $f_{H_1}$ is a $d$-embedding of $H_1$, and further, $f_{H_1}\in W=U_1\times U_2\times\cdots\times U_{n_1}$. Since $H_1$ is a generically $d$-rigid, $R\mathcal{E}(H_1)$ is an open and dense subset of the set of all $d$-embeddings $\mathcal{E}(H)$. Since $W$ is an open subset of $\mathcal{E}(H)$, $R\mathcal{E}(H_1)\cap W \neq \emptyset$. Let $\phi\in R\mathcal{E}(H_1)\cap W$. Then $\phi$ is a rigid $d$-embedding of $H_1$. Define a map $\Phi : V(G)\to \mathbb{R}^d$ in the following way:

\begin{eqnarray*}
\Phi (v_i)&=& \phi(v_i)\hspace{.5cm}\text{if}\hspace{.2cm} 1\leq i\leq n_1, \\
  &=& f(v_i) \hspace{.5cm}\text{elsewhere}.
\end{eqnarray*}
Then $\Phi\in U$, and therefore, $\Phi$ is a generically $d$-rigid embedding of $G$ such that the restriction of $\Phi$ on $V(H_1)$ is a rigid $d$-embedding of $H_1$.

Let $Z\subseteq (R\mathcal{E}(H_1)\cap W)\times U_{n_1+1}\times\dots\times U_{|V|}$ be a basic open set in $\mathbb{R}^{d\cdot |V|}$ containing $\Phi$, where $Z=Z_1\times\dots\times Z_{|V|}$. Then every element $h$ in $Z$ is a rigid $d$-embedding of $G$, whose restriction on $V(H_1)$ is a rigid $d$-embedding of $H_1$. Let $Z_{i_1},\dots,Z_{i_{n_2}}$ correspond to the vertices of $H_2$, and take $W_1=Z_{i_1}\times\dots\times Z_{i_{n_2}}$. Then the restriction of $\Phi$ on $W_1$ is a $d$-embedding of $H_2$. Since $H_2$ is generically $d$-rigid and $W_1$ is an open set containing a $d$-embedding of $H_2$, we have  $R\mathcal{E}(H_2)\cap W_1 \neq \emptyset$. Let $\psi\in R\mathcal{E}(H_2)\cap W_1$, i.e., $\psi$ is a rigid $d$-embedding of $H_2$. Now, define a map $\Psi:\, V(G)\to\mathbb{R}^d$ in the following way:
\begin{eqnarray*}
\Psi (v_i)&=& \psi(v_i)\hspace{.5cm}\text{if}\hspace{.2cm}i_1\leq i\leq i_{n_2}, \\
  &=& \Phi(v_i) \hspace{.5cm}\text{elsewhere}.
\end{eqnarray*}
The map $\Psi$ is an element of $Z$, and the restriction of $\Psi$ on $V(H_2)$ is a rigid $d$-embedding of $H_2$. This completes the proof.
\end{proof}

\begin{Corollary}\label{Many rigid subgraph}
Let $G$ be a generically $d$-rigid graph, and let $H_1,\dots,H_k$ be generically $d$-rigid subgraphs of $G$. Then there is a rigid $d$-embedding $\Psi$ of $G$, such that the restriction of $\Psi$ on $V(H_i)$ is a rigid $d$-embedding of $H_i$, for $1\leq i\leq k$.
\end{Corollary}
For a $d$-embedding $f$ of a graph $G$, a stress of $G$ corresponding to $f$ is a function $\omega : E(G)\to\mathbb{R}$ such that for every vertex $v\in V(G)$,
\begin{eqnarray*}
\sum_{uv\in E(G)} \omega(uv)[f(v)-f(u)]=0.
\end{eqnarray*}
We say that an edge $uv$ participates in a stress $\omega$ if $\omega(uv)\neq 0$, and that a vertex $v$ participates in $\omega$ if there exists a vertex $u$ such that the edge $uv$ participates in $\omega$. The collection of all stresses of $G$ corresponding to a given $d$-embedding $f$ forms a real vector space called the {\em stress space} of $G$ corresponding to $f$. Let us denote the stress space of $G$ corresponding to $f$ by $\mathcal{S}(G_f)$. This stress space can be described in terms of matrices as well. Let $|V|=n$ and $|E|=m$. Let $Rig(G,f)$ be the $m \times dn$ matrix whose rows are labeled by edges, and the columns are grouped into blocks of size $d$ labeled by vertices. In the row corresponding to the edge $e=uv$, the $d$-vector occupied in the block column corresponding to the vertex $u$ (resp. $v$) will be $f(u)-f(v)$ (resp. $f(v)-f(u)$). The rest of the columns in that row will be occupied by the zero vector. The matrix $Rig(G,f)$ is called the {\em rigidity matrix} of the graph $G$ corresponding to the $d$-embedding $f$. It is easy to see that the stress space of a graph $G$ corresponding to a $d$-embedding $f$ is the same as the left kernel of the rigidity matrix $Rig(G,f)$. If $G$ is a generically $d$-rigid graph, then the elements in $R\mathcal{E}(G)$ are referred to as generic maps. Moreover, in the case of a generically $d$-rigid graph $G$, the dimension of the stress space remains independent of the choice of generic maps \cite{Kalai}. In such instances, we denote the rigidity matrix of the graph $G$ as $Rig(G, d)$ and the corresponding stress space as $\mathcal{S}(G)$. By a generic $d$-stress (or generic stress when $d$ is understood) of a generically $d$-rigid graph $G$, we mean a stress of $G$ corresponding to a generic map. For a generically $d$-rigid graph $G$, define $\gamma(G):=m-dn+{d+1\choose 2}$.
%
%

\begin{Proposition}{\rm\cite{AshimowRoth, Kalai}}\label{g2 as dimension of stress space}
Let $G=(V,E)$ be a generically $d$-rigid graph. Then $\gamma(G)$ is the dimension of the left kernel of the rigidity matrix $Rig(G,d)$. 
\end{Proposition}

\begin{Proposition}{\rm\cite[Cone Lemma]{TayWhiteWhiteley,Whiteley}}\label{Cone Lemma}
Let $G$ be a given graph, and let $C(G,u)$ be the cone over $G$ with a new vertex $u$. Then,
\begin{enumerate}[$(i)$]
 \item $G$ is generically $d$-rigid if and only if $C(G,u)$ is generically $(d+1)$-rigid.
 \item The left null spaces of $Rig(G,d)$ and $Rig(C(G,u),d+1)$ are isomorphic real vector spaces. Moreover, if $\gamma(G)\neq 0$, then $u$ participates in a generic stress of $C(G,u)$.
 \end{enumerate}
\end{Proposition}

A simplicial complex $\D$ is called {\em generically $d$-rigid} if its graph (or 1-skeleton), denoted as $G(\D)$, is generically $d$-rigid. We will use the notations $\mathcal{S}(\D_f)$ and $\mathcal{S}(\D)$ to refer to $\mathcal{S}(G(\D)_f)$ and $\mathcal{S}(G(\D))$, respectively. Note that, if $\D$ is a $(d-1)$-dimensional simplicial complex such that $G(\D)$ is generically $d$-rigid, then $g_2(\D)=\gamma(G(\D))$. The notion of rigidity is applied in the structural analysis of simplicial complexes, particularly in various contexts involving normal pseudomanifolds.

\begin{Proposition}{\rm \cite{Fogelsanger,Kalai}}
Let $d\geq 3$, and let $\D$ be a normal $d$-pseudomanifold. Then,
\begin{enumerate}[$(i)$]
 \item $\D$ is generically $(d+1)$-rigid, and
 \item for every face $\s$ of co-dimension $3$ or more, $g_2(\lk\s)\leq g_2(\D)$.
 \end{enumerate}
\end{Proposition}
\begin{Proposition}\label{complete graph} 
 Let $\D$ be a  normal $d$-pseudomanifold with $g_2(\D)=p$, and $g_2(\lk v)\geq 1$ for every vertex $v\in\D$. Let $u$ be a vertex such that $g_2(\D[V(\st u)])=p-1$. Then the graph $G(\D[V(\D\setminus\st u)])$ is complete.
\end{Proposition}
\begin{proof}
If possible, let $v$ and $z$ be two vertices in $V(\D\setminus\st u)$ such that $vz\notin\D$. Note that $\D[V(\st u)]$, $\st v$, and $\st z$ are generically $(d+1)$-rigid. It follows from Corollary \ref{Many rigid subgraph} that there is a rigid $(d+1)$-embedding of $G(\D)$ such that the restrictions of the map on $V(\st u), V(\st v)$ and $V(\st z)$ are rigid $(d+1)$-embeddings of $G(\D[V(\st u)])$, $G(\st v)$ and $G(\st z)$, respectively. Let $f$ be such a rigid $(d+1)$-embedding.

 Since $g_2(\D[V(\st u)])=p-1$, the dimension of the stress space  $\mathcal{S}(\D[V(\st u)]_f)$ is $p-1$. Let $\{w_1,\dots,w_{p-1}\}$ be a basis for $\mathcal{S}(\D[V(\st u)]_f)$, where $w_i:E(\D[V(\st u)])\to\mathbb{R}$ are stresses. Let us extend these stresses $w_1,\dots,w_{p-1}$ to $\tilde{w}_1,\dots,\tilde{w}_{p-1}$, respectively, on $E(\D)$ by taking $\tilde{w}_i(e)=0$ for all $e\in E(\D)\setminus E(\D[V(\st u)])$. Note that $\tilde{w}_1,\dots,\tilde{w}_{p-1}$ are linearly independent in $\mathcal{S}(\D_f)$.

  Since $g_2(\lk v)\geq 1$, by Cone Lemma, the vertex $v$ participates in a nonzero stress in $\mathcal{S}(\st v_f)$. Let $w_p:E(\st v)\to\R$ be such a nonzero stress. Now extend this stress $w_p$ to a stress $\tilde{w}_p:E(\D)\to\R$ by taking $\tilde{w}_p(e)=0$ for every edge $e\in E(\D)\setminus E(\st v)$. Then $\tilde{w}_1,\dots,\tilde{w}_{p-1}$ and $\tilde{w}_p$ are independent in $\mathcal{S}(\D_f)$ and hence a basis of $\mathcal{S}(\D_f)$. 

 By the same argument as in the previous paragraph, there is a nonzero stress $w_{p+1}:E(\st z)\to\R$ such that $z$ participates in $w_{p+1}$. Since $v\notin\st z$, the vertex $v$ does not participates in the stress $w_{p+1}$. Let us extend this stress $w_{p+1}$ to $\tilde{w}_{p+1}:E(\D)\to\R$ by taking $\tilde{w}_{p+1}(e)=0$ for every edge $e\in E(\D)\setminus E(\st z)$. Then $\tilde{w}_{p+1}$ cannot be written as a linear combination of $\tilde{w}_1,\dots,\tilde{w}_{p}$ as $\tilde{w}_i(e)=0$ for every edge $e$ incident to $z$. Therefore,  $\tilde{w}_1,\dots,\tilde{w}_p$ and $\tilde{w}_{p+1}$ are independent in $\mathcal{S}(\D_f)$, which contradicts the fact that $g_2(\D)=p$. Thus,  $vz\in\D$ for any two vertices $v$ and $z$ in $\D\setminus\st u$, i.e., $G(\D[V(\D\setminus\st u)])$ is complete. 
\end{proof}
\begin{Proposition}\label{same vertex set}
Let $\D$ be a  normal $d$-pseudomanifold such that $g_2(\lk v)\geq 1$ for every vertex $v\in\D$. If $u$ is a vertex in $\D$ with $g_2(\D)=g_2(\D[V(\st u)])$, then $V(\D)=V(\st u)$.
\end{Proposition}
\begin{proof}
If possible, let $v$ be a vertex in $V(\D\setminus\st u)$. Note that $\D[V(\st u)]$ and $\st v$ are generically $(d+1)$-rigid. By Corollary \ref{Many rigid subgraph}, there is a rigid $(d+1)$-embedding, say $f$, of $G(\D)$ such that the restrictions of the map on $V(\st u)$ and $ V(\st v)$ are rigid $(d+1)$-embeddings of $G(\D[V(\st u)])$ and $G(\st v)$, respectively.

 Let $g_2(\D)=p$. Then the stress  space $\mathcal{S}(\D[V(\st u)]_f)$ has dimension $p$. By a similar argument as in Proposition \ref{complete graph},  there exist linearly independent stresses $\tilde{w}_1,\dots,\tilde{w}_{p}$ in $\mathcal{S}(\D_f)$ such that $\tilde{w}_i(e)=0$ for all $e\in E(\D)\setminus E(\D[V(\st u)])$. Moreover, $g_2(\lk v)\geq 1$ implies that there is a nonzero stress $\tilde{w}_{p+1}$ in $\mathcal{S}(\D_f)$ such that the vertex $v$ participates in $\tilde{w}_{p+1}$, and $\tilde{w}_p(e)=0$ for every edge $e\in E(\D)\setminus E(\st v)$.  Therefore,  $\tilde{w}_1,\dots,\tilde{w}_p$ and $\tilde{w}_{p+1}$ are independent in $\mathcal{S}(\D_f)$, which contradicts the fact that $g_2(\D)=p$. This completes the proof. 
%
%
\end{proof}
  
\subsection{Combinatorial operations}
\begin{definition}
{\rm Let $\D$ be a normal $4$-pseudomanifold, and let $ab$ be an edge in $\D$ such that $\lk ab=\p(cd)\star\p(xyz)$, where $cd$ is a missing edge in $\D$. Consider the simplicial complex $\D':=(\D\setminus ab\star\p(cd)\star\p(xyz)) \cup \p(ab)\star cd\star\p(xyz)$. We say that $\Delta'$ is obtained from $\D$ by {\em flipping} the edge $ab$ with $cd$. }
\end{definition}

\begin{definition}
{\rm Let $\D$ be a normal $4$-pseudomanifold, and let $\s$ be a $(4-i)$-simplex in $\D$, where $1\leq i\leq 3$. If $\lk \s =\p(\t)$, where $\t$ is an $i$-simplex and $\t\notin\D$, then consider the simplicial complex $\D':=(\D\setminus \s\star\p(\t))\cup (\t\star\p(\s))$. We say that the complex $\D'$ is obtained from $\D$ by a \textit{bistellar $i$-move} with respect to the pair $(\s,\t)$.}
\end{definition}

\begin{definition}\label{contraction-expansion}
%
{\rm Let $\D$ be a normal $4$-pseudomanifold, and let $w$ be a vertex in $\D$. Consider a $2$-dimensional sphere $S$ in $\lk w$ that separates $\lk w$ into two portions, $D_1$ and $D_2$, where $D_1$ is a triangulated 3-ball, and $D_1$ and $D_2$ have common boundary complex $S$. Now, let $\D':=(\D\setminus \st w)\cup (uv\star S)\cup (u\star D_1)\cup (v\star D_2)$, where $u$ and $v$ are new vertices and $uv$ is a new edge. Then $lk_{\D'} u \cap lk_{\D'} v= lk_{\D'} uv$, and we say that $\D$ is obtained from $\D'$ by {\em contracting} the edge $uv$ to the vertex $w$. The normal pseudomanifold $\D'$ is said to be obtained from $\D$ by an {\em edge expansion}
}
\end{definition}

In particular, if there is a vertex, say $x$, in $\lk w$ such that $\lk xw$ is the $(d-2)$-dimensional sphere $S$, then we say $\D'$ is obtained by a \textit{central retriangulation} of $\D$ along $\st xw$ with center $u$ (or by stellar subdivision at the edge $xw$ with the vertex $u$, see \cite{Swartz2009} for more details). Thus, the central retriangulation of a normal pseudomanifold along the star of an edge (or the stellar subdivision at an edge) is a special case of an edge expansion.

\begin{definition}
{\rm Let $u$ and $v$ be two vertices in a normal 4-pseudomanifold $K$, such that either $lk_{K} u$ or $lk_{K} v$ is a triangulated sphere. We say that the pair $(u, v)$ satisfies the \textit{link condition} if $uv$ is an edge in $K$, and $lk_{K} u\cap lk_{K} v=lk_{K} uv$ holds.}
\end{definition}

Consider a pair $(u, v)$ of vertices $u$ and $v$ in a normal 4-pseudomanifold $K$ that satisfies the link condition. We can contract the edge $uv$ in this case. If $K'$ is the resulting complex, then the geometric carriers of $K'$ and $K$ are PL-homeomorphic (cf. \cite{BGS1}).


%

\begin{Lemma}\label{missing pqc}
Let $\D$ be a homology $4$-manifold,  and let $u$ be a vertex in $\D$ such that $\lk u=C\star\p(\s^2)$, where $\s^2$ is a $2$-simplex,  and $C$ is a cycle of length $n$. If $\s^2\notin\D$, then $\D$ is obtained from a homology $4$-manifold $\D'$ with $g_2(\D') = g_2(\D)-(n-2)$ by an edge expansion.
\end{Lemma}
\begin{proof}
Let $\s^2=pqc$. Then, $\lk u\cap\lk p=\lk up=C\star\p(qc)$. Therefore, the pair $(u,p)$ satisfies the link condition, and we can contract the edge $up$. Let $\D'$ be the resulting complex. Then $f_0(\D')=f_0(\D)-1$, and $f_1(\D')=f_1(\D)-(n+3)$. Thus,  $g_2(\D')= f_1(\D)-(n+3)-5(f_0(\D)-1)+15=g_2(\D)-(n-2)$, and $\D$ is obtained from $\D'$ by an edge expansion.
\end{proof}
\begin{Lemma}\label{missing edge}
Let $\D$ be a homology $4$-manifold, and let $u$ be a vertex in $\D$ such that $\lk u=C\star\p(\s^2)$, where $\s^2$ is a $2$-simplex, and $C$ is a cycle of length $n \geq 4$. If $C$ contains a vertex that is incident to no diagonal edges in $\D[V(C)]$, then $\D$ is obtained from a homology $4$-manifold $\D'$ with $g_2(\D') = g_2(\D)-1$ by an edge expansion.
\end{Lemma}
\begin{proof}
Let $x\in C$ be a vertex that is free from all the diagonal edges of $\D[V(C)]$. Then $\lk u\cap\lk x=\lk ux=\p(\s^2)\star\p(x_1x_2)$, where $x_1$ and $x_2$ are two vertices of $C$ adjacent to $x$ in $C$. Let $\D'$ be the complex obtained by contracting the edge $ux$. Then $g_2(\D')=g_2(\D)-1$, and $\D$ is obtained from $\D'$ by an edge expansion. 
\end{proof}

%
%

 
 \begin{Lemma}\label{non prime vertex links}
Let  $\D$ be a prime homology $4$-manifold, and let $u$ be a vertex in $\D$ such that $\lk u$ is not prime. Then $\D$ is obtained from a homology $4$-manifold $\D'$ with $g_2(\D')=g_2(\D)-1$ by applying a bistellar 1-move and an edge contraction.
\end{Lemma}
\begin{proof}
Since $\lk u$ is not prime, we have $\lk u=S_1\#_{\s}S_2$, where $S_1$ and $S_2$ are homology $3$-spheres and $\s$ is a missing $3$-simplex in $\lk u$. Since $\D$ is prime, $\s\notin\D$. Let $\D'= (\D\setminus\{\t:u\leq \t\})\cup x_1\star (S_1\cup\{\s\})\cup  x_2\star (S_2\cup\{\s\})$, where $x_1$ and $x_2$ are two new vertices. Then $g_2(\D')=g_2(\D)-1$, and $\D$ is obtained from $\D'$ by applying a bistellar 1-move and an edge contraction.
\end{proof}
 
 The operation employed in Lemma \ref{non prime vertex links}, which involves replacing $\{\t:u\leq \t\}$ with $x_1\star (S_1\cup\{\s\})\cup  x_2\star (S_2\cup\{\s\})$ in $\D$, was originally defined in \cite{Swartz2008, Walkup}. In this context, we have used the reverse of that operation, which is a combination of a bistellar 1-move and an edge contraction.
 
 \begin{definition}
 {\rm For $\a\in\mathbb{N}$, the class $\H_{\a}$ consists of all homology $4$-manifolds $\D$ such that $g_2(\D)=\a$, and  $\D$ satisfies the following conditions:
 \begin{enumerate}[$(i)$]
 \item  $\D$ is prime, i.e., if $\D$ contains the boundary complex $\p(\s^4)$ of a $4$-simplex $\s^4$, then $\s^4\in\D$,
 \item $\lk v$ is prime for every vertex $v\in\D$,
 \item if $u$ is a vertex in $\D$ such that $\lk u=C\star\p(\s^2)$, where $\s^2$ is a $2$-simplex and $C$ is a cycle of length $n\geq 4$, then $\s^2\in\D$, and every vertex of $C$ is incident to at least one diagonal edge in $\D[V(C)]$.
 \end{enumerate}
 }
 \end{definition}
 
If a homology $4$-manifold $\D$ contains a vertex that has a non-prime link, then the structure of $\D$ can be deduced from Lemma \ref{non prime vertex links}. In a similar spirit, if $\D$ contains a vertex, say $u$, such that $\lk u = C\star \p(\s^2)$, where $C$ is a cycle of length $n\geq 4$, $\s^2$ is a 2-simplex, and either $\s^2$ is a missing $2$-simplex in $\D$, or there is a vertex in $C$ that is not incident to a diagonal edge in $\D[V(C)]$, then we can characterize the structure of $\D$ from Lemmas \ref{missing pqc} and \ref{missing edge}. Hence, considering Lemmas \ref{missing pqc}, \ref{missing edge}, and \ref{non prime vertex links}, it is sufficient to focus on analyzing the class $\H_{\a}$ in order to describe all homology $4$-manifolds with $g_2\leq 5$. It follows from Lemma \ref{> 1 for prime u} that $g_2(\lk v)\geq 1$ for every vertex $v$ in a homology $4$-manifold $\D$ belonging to the class $\H_{\a}$.
 
 \begin{Lemma}\label{g2<3 exists}
Let $\D\in\H_{\a}$, where $\a\leq 5$. Then there is a vertex $u\in\D$ with $g_2(\lk u)\leq 2$.
\end{Lemma}
\begin{proof}
If possible, let $g_2(\lk v)\geq 3$ for every vertex $v\in\D$. Then by Lemmas \ref{g-relations} and \ref{g3=0}, $3f_0(\D)\leq 0 + 4\cdot 5=20$, i.e., $f_0(\D)\leq 6$. Since $\D$ is a homology 4-manifold, it must be a stacked sphere. However, this contradicts the fact that $g_2(\D)\geq 1$. Thus, the result follows.
\end{proof}

\begin{Lemma}\label{octahedral 4-sphere}
Let $\D\in\H_{\a}$, where $4\leq\a\leq 5$. If $u$ is a vertex in $\D$ such that $\lk u$ is an octahedral $3$-sphere, then $\D$ is obtained by an edge expansion from a homology $4$-manifold with $g_2\leq 3$.
\end{Lemma}
\begin{proof}
Since $\lk u$ is an octahedral 3-sphere, it is prime with $g_2(\lk u)=2$.  Given that $4\leq g_2(\D)\leq 5$, there must be at least two antipodal vertices, say $x$ and $y$, in $\lk u$ that are not adjacent in $\D$. Consider the octahedral $2$-sphere $S = \D[V(\lk u)\setminus\{x,y\}]$ and the 4-dimensional combinatorial ball $B=xy\star S$. Then $\p B=\lk u$. Let $\D'$ be the homology $4$-manifold obtained from $\D$ by replacing $\st u$ with $B$. Then $g_2(\D')=g_2(\D)-2$ and $\D$ is obtained  by a central retriangulation of $\D'$ along $st_{\D'}xy$. This completes the proof.
\end{proof}
\section{4-dimensional Homology Manifolds with $g_2=3$}
\begin{Lemma}\label{g2=3, g2(lk u)<2 exists}
Let $\D$ be a homology $4$-manifold with $g_2(\D)=3$. Then there is a vertex $u$ in $\D$ with $g_2(\lk u)\leq 1$.
\end{Lemma}
\begin{proof}
 It follows from Lemma \ref{g3=0} that $g_3(\D)=0$. Suppose $g_2(\lk v)\geq 2$ for every vertex $v\in\D$. Then Lemma \ref{g-relations} implies that $2f_0(\D)\leq 3\cdot4=12$, i.e., $f_0(\D)\leq 6$. This contradicts the fact that $g_2(\D)=3$. Thus, there exists a vertex $u\in\D$ with $g_2(\lk u)\leq 1$.
 \end{proof}
 \begin{Lemma}\label{g2=3, g2(lku)=1, d(u)>=7}
Let $\D\in\H_{3}$, and let $u$ be a vertex in $\D$ with $g_2(\lk u)=1$. Then  $d(u) = 6$. Moreover, $\D$ is obtained from a triangulated sphere with $g_2\leq 2$ by an edge expansion.
\end{Lemma}

 \begin{proof}
If possible, let $\lk u=C\star\p(pqc)$, where $pqc$ is a 2-simplex and $C$ is a cycle of length $n\geq 4$. Since $g_2(\D)=3$, $\lk u$ contains at most two missing edges that are present in $\D$. Since every vertex of $C$ is incident to at least one diagonal edge, we have $n=4$. Therefore, by Proposition \ref{same vertex set}, $V(\D)=V(\st u)$. However, $pqc\in\D$ implies that $V(\lk pqc)\cap V(C)\neq\emptyset$. Let $a\in V(\lk pqc)\cap V(C)$. Then, $\p(upqc)\subseteq\lk a$. However, $upqc\notin\D$, which contradicts the fact that $\lk a$ is prime. Hence, $d(u)=6$.

Let $\lk u=\p(abt)\star\p(pqc)$, where $abt$ and $pqc$ are two $2$-simplices. Suppose both $abt$ and $pqc$ are faces of $\D$. By similar arguments as in the last paragraph, $\{a,b,t\}\cap V(\lk pqc)=\emptyset$. Therefore, $V(\lk pqc)\subseteq V(\D\setminus\st u)$. Similarly, $V(\lk abt)\subseteq V(\D\setminus\st u)$. Hence, $d(x)\geq 8$ for every vertex $x$ in $\lk u$. Therefore, using the arguments as in the last paragraph, we have $g_2(\lk x)\geq 2$ for every vertex $x$ in $\lk u$. However, the equations $\sum_{v\in\D} g_2(\lk v)=3g_3(\D)+4g_2(\D)$ and $g_3(\D)=0$ imply that $6\cdot 2+1<0+3\cdot4$, which is not possible. Thus, at least one of $abt$ and $pqc$ is a missing triangle in $\D$. Therefore, the result follows from Lemma \ref{missing pqc}.
 \end{proof}
\begin{Theorem}\label{prime theorem for g2=3}
Let $\D$ be a prime homology $4$-manifold with $g_2(\D)=3$. Then $\D$ is obtained from a triangulated sphere with $g_2\leq 2$ by either an edge expansion or a bistellar $1$-move and an edge contraction.
\end{Theorem}
\begin{proof}
If $\D\notin\H_3$, then the result follows from Lemmas \ref{missing pqc}, \ref{missing edge}, and \ref{non prime vertex links}. Now, assume that $\D\in \H_3$. It follows from Lemma \ref{g2=3, g2(lk u)<2 exists} that there is a vertex $v$ in $\D$ with $g_2(\lk v)\leq 1$. For every vertex $y\in\D$, we have $g_2(\lk y)\geq 1$, implying the existence of a vertex, say $u$, in $\D$ such that $g_2(\lk u)=1$. Therefore, we can conclude the result from Lemma \ref{g2=3, g2(lku)=1, d(u)>=7}.
\end{proof}

\section{4-dimensional Homology Manifolds with $g_2=4$}
\begin{Lemma}\label{d=4, g2(lku)=1, d(u)>=8}
Let $\D\in\H_{4}$. If $u$ is a vertex in $\D$ with $g_2(\lk u)=1$, then $d( u) = 6$.
\end{Lemma}
\begin{proof}
If possible, let $d( u)\geq 7$. Consider $\lk u=C\star\p(pqc)$, where $pqc$ is a 2-simplex, and $C$ is a cycle of length $n\geq 4$. Here, $pqc\in\D$, and every vertex of $C$ is incident to at least one diagonal edge in $\D[V(C)]$. 

Suppose $n\geq 5$, which implies that $d(u)\geq 8$. Since $g_2(\D)=4$, $\lk u$ contains exactly three missing edges that are present in $\D$. Therefore, by Proposition \ref{same vertex set}, $V(\D)=V(\st u)$. Since $pqc\in\D$, $\lk pqc$ is a cycle containing at least three vertices of $C$. Let $s\in V(C)\cap V(\lk pqc)$. Then, it follows that $\partial(upqc) \subseteq \lk s$. However, $upqc \notin \D$, which contradicts the fact that $\lk s$ is prime. Therefore, $n=4$, and we have $d(u)= 7$.


Let $C=C(a,b,d,t)$ as illustrated in Figure \ref{fig:1}. By a similar argument to the previous paragraph, we find that $V(\lk pqc)\subseteq V(\D\setminus \st u)$. Furthermore, $\lk u$ contains precisely two missing edges that are present in $\D$, as indicated in Figure \ref{fig:1} with dotted lines. We claim that $d(k)\geq 8$ for every vertex $k\in C$.

\begin{figure}[ht]
\tikzstyle{ver}=[]
\tikzstyle{vertex}=[circle, draw, fill=black!100, inner sep=0pt, minimum width=4pt]
\tikzstyle{edge} = [draw,thick,-]
\centering

\begin{tikzpicture}[scale=1.2]
\begin{scope}[shift={(0,0)}]
\foreach \x/\y/\z in {0/0/d,2/0/b,2/1.6/a,0/1.6/t,3.5/0/q,6/0/c,4.5/1.6/p}{
\node[vertex] (\z) at (\x,\y){};}
\foreach \x/\y in {a/b,a/t,p/q,q/c,p/c,d/t,b/d}{\path[edge] (\x) -- (\y);}
\foreach \x/\y in {a/d,b/t}{\draw[densely dotted] (\x) -- (\y);}
\foreach \x/\y/\z in {-.3/-.1/d,2.3/-0.1/b,2.25/1.8/a,-.3/1.8/t,3.2/-.1/q,6.2/-0.2/c,4.5/1.9/p}
\node[ver] () at (\x,\y){$\z$};
\end{scope}
\end{tikzpicture}
\caption{} \label{fig:1}
\end{figure}

 Without loss of generality, let $d(a)=7$. If $\lk ad$ contains exactly $4$ vertices, then it is the boundary complex of a $3$-simplex. Since $ad\notin\lk pqc$, $\lk ad$ contains precisely two vertices from $\{p,q,c\}$. Without loss of generality, let $\{p,q\}\subseteq \lk ad$. Then $\lk ad=\p(btpq)$,  and consequently,  $d\star\p(btpq)\subseteq\lk a$.  As $\lk a $ is prime, $\p(btpq)\subseteq\lk a$ implies $btpq\in\lk a$.  Therefore, $\lk a=\p(dbtpq)$, contradicting the fact that $c\in\lk a$. Thus, $V(\lk ad)=\{p,q,c,b,t\}$, and $\lk ad$ is one of the following:

\begin{enumerate}[$(i)$]
\item $\p(pqc)\star\p(bt)$,
\item $\p(btc)\star\p(pq)$, $\p(btp)\star\p(cq)$, or $\p(btq)\star\p(pc)$.
\end{enumerate}

If $\lk ad$ is of type $(i)$, then $\lk pq$ contains the triangles $adt$ and $adb$. However, $\lk u = C \star \partial(pqc)$ implies that $u\star C \subseteq \lk pq$. Since $\lk pq$ is closed, the maximal simplices of $\lk pq$ consist of the 2-simplices from $(u \star C) \cup \{adt, adb\}$, which contradicts the fact that $c \in \lk pq$. Now, let $\lk ad$ be of type $(ii)$. Without loss of generality, assume that $\lk ad = \partial(btc) \star \partial(pq)$. Then, $adb$ and $adt$ belong to $\lk pc$. Furthermore, $\lk u = C \star \partial(pqc)$ implies that the maximal simplices of $\lk pc$ consist of the 2-simplices from $(u\star C) \cup \{adt, adb\}$. This contradicts the fact that $q \in \lk pc$. Therefore, $d(a) \geq 8$. Using similar arguments, we can conclude that $d(k) \geq 8$ for every vertex $k \in C$. 

Since $V(\lk pqc) \subseteq V(\D \setminus \st u)$, $\D \setminus \st u$ must contain at least three vertices. Therefore, for every vertex $x\in\lk u$, we have $d(x)\geq 8$. Hence, the arguments in the second paragraph of the proof imply that $g_2(\lk x)\geq 2$. Furthermore, we can calculate that $\sum_{v \in \D} g_2(\lk v) = \sum_{x \in \lk u} g_2(\lk x) + g_2(\lk u) + \sum_{y \in \D \setminus \st u} g_2(\lk y) \geq 18$. However, $g_3(\D) = 0$ implies that $\sum_{v \in \D} g_2(\lk v) \leq 4 \cdot g_2(\D) = 16$, which creates a contradiction. Hence, $d(u)=7$ is not possible. This completes the proof.
\end{proof}
\begin{Lemma}\label{d=4, g2(lku=2), d(u)>=8, no missing edge}
Let $\D\in\H_{4}$, and let $u$ be a vertex in $\D$ with $g_2(\lk u)=2$. If $G(\D[V(\lk u)])=G(\lk u)$, then $\D$ is obtained from a triangulated sphere with $g_2\leq 3$ by an edge expansion.
\end{Lemma}
\begin{proof}
If $\lk u$ is an octahedral 3-sphere, then the result follows from Lemma \ref{octahedral 4-sphere}. Now, assume that $\lk u$ is not an octahedral $3$-sphere. According to Lemma \ref{Zheng g2=2}, $\lk u$ arises from a triangulated $3$-sphere $K$ with $g_2=1$ by subdividing a triangle. Lemma \ref{Nevo g2=1} implies that $K$ results from finitely many facet subdivisions of a prime triangulated $3$-sphere $K' = C\star\partial(pqc)$, where $C$ is a cycle of length $n\geq 3$ and $pqc$ is a $2$-simplex. If $K$ is prime, then $K = K'$.
Let $\lk u$ be derived from $K$ by subdividing a triangle, denoted $\tau^2$, with a new vertex $w$. The cycle $C$ or the triangle $pqc$ contains exactly one vertex of the triangle $\tau^2$, designated as the central vertex of $\tau^2$. Moreover, $\tau^2$ is a proper face of each facet that participated in the subdivision process; otherwise, it would contradict the fact that $\lk u$ is prime.

Let $K$ not be prime. Since $G(\lk u)=G(\D[V(\lk u)])$, the vertex $w$ is not incident to any missing edge in $\lk u$ that is present in $\D$. Consequently, $\lk u\cap\lk w=\lk uw=\p(\t^2)\star\p(rs)$, where $lk_{K}\t^2=\p(rs)$, and we can contract the edge $uw$. Let $\D'$ be the resulting complex. Then $g_2(\D')=3$, and $\D$ is obtained from $\D'$ through an edge expansion.

Suppose $K$ is prime. Then $K=K'=C\star\p(pqc)$, where $C$ is a cycle of length $n$, and $pqc$ is a 2-simplex, as mentioned earlier. Let $n\geq 4$. Since $\lk u$ contains no missing edges that are present in $\D$, there is a vertex $z\in C$ that is not adjacent to $w$. Consequently, for the vertex $z$, we have $\lk u\cap\lk z=\lk uz=\p(pqc)\star\p(xy)$, where $x$ and $y$ are the neighboring vertices of $z$ in $C$. Thus, we can contract the edge $zw$, and therefore, we are done. 

Now, consider the case where $n=3$, and let $K=\p(abt)\star\p(pqc)$, where $abt$ is a 2-simplex. Without loss of generality, let $\t^2=pqt$. In this case, $wc\notin\D$. If $pqc\in\D$, then $\lk pqc$ contains at least two elements from $V(\D\setminus\st u)$, implying $d(x)\geq 7$ for each vertex in $\{a,b,t,p,q,c\}$. Thus, $g_2(\lk x)\geq 2$ for every vertex $x\in\{a,b,t,p,q,c\}$. From Lemma \ref{g-relations}, we conclude that $V(\D\setminus\st u)$ contains at most one element, which contradicts the fact that $\lk pqc$ contains at least two elements from $V(\D\setminus\st u)$. Therefore, $pqc$ is a missing triangle in $\D$, and the pair $(u,c)$ satisfies the link condition. Let $\D'$ be the complex obtained by contracting the edge $uc$. Then $g_2(\D') = 3$, and $\D$ is obtained from $\D'$ by an edge expansion.
\end{proof}

\begin{Lemma}\label{d=4, g2(lku=2), d(u)>=8, one missing edge}
Let $\D\in\H_{4}$, and let $u$ be a vertex in $\D$ with $g_2(\lk u)=2$. If $\lk u$ contains at least one missing edge that is present in $\D$, then $\D$ is obtained from a triangulated sphere with $g_2\leq 3$ by an edge expansion. 
\end{Lemma}
\begin{proof}
If $\lk u$ is an octahedral 3-sphere, then the result follows from Lemma \ref{octahedral 4-sphere}. Now, assume that $\lk u$ is not an octahedral 3-sphere. Let $K,K',\t^2$, and $w$ be as described in Lemma \ref{d=4, g2(lku=2), d(u)>=8, no missing edge}. 

\vst
\noindent\textbf{Case 1:} Suppose $K$ is not prime, and let $lk_{K} \t^2=\{r,s\}$,  where $s$ is the vertex added to $K'$ at the last facet subdivision. The cycle $C$, or the triangle $pqc$, contains exactly one vertex of the triangle $\t^2$, which we refer to as the central vertex of $\t^2$. Note that if $w$ is not incident to any missing edge of $\lk u$ that is present in $\D$, then the result follows from a similar argument as in Lemma \ref{d=4, g2(lku=2), d(u)>=8, no missing edge}. For the remaining part of the proof, we assume that $w$ is incident to a missing edge in $\lk u$ that is present in $\D$. 

Let $S=\{s_1,\dots,s_m\}$ be the set of vertices added during the facet subdivisions of $K'$ where $s=s_m$. If each vertex in $S$ is incident to a missing edge in $\lk u$ that is present in $\D$, then every vertex in $S\cup V(\st w)$ has a degree of at least 7, and hence $g_2(\lk x)\geq 2$ for every vertex $x\in S\cup V(\st w)$. Since $g_3(\D)=0$, according to Lemma \ref{g-relations}, we have $f_0(\D)\leq 8$. However, since $K$ is not prime, $d(u)\geq 8$, and hence $f_0(\D)\geq 9$. This gives a contradiction. Let $s_j$ not be incident to any missing edge of $\lk u$ that is present in $\D$. Then $\lk s_j\cap\lk u=\lk us_j=\p(\t^2)\star\p(yz)$, where $y$ and $z$ are the vertices incident to $s_j$ in $\lk u$. Thus,  we can contract the edge $us_j$, and we are done.

\vst
\noindent\textbf{Case 2:}
Suppose $K$ is prime, i.e., $K=C\star\p(pqc)$, where $C$ is a cycle of length $n\geq 3$. Let $n=3$ and $K=\p(abt)\star\p(pqc)$, where $abt$ is a 2-simplex. Without loss of generality, we assume that $\t^2=pqt$. Then $wc$ is the missing edge in $\lk u$, which is present in $\D$. Therefore, $d(x)\geq 7$, and hence $g_2(\lk x)\geq 2$ for every vertex $x\in\lk u$. Now, Lemma \ref{g-relations} implies that $V(\D\setminus \st u)=\emptyset$. Therefore, $abt$, $pqc$, and $pqt$ are missing triangles in $\D$. Hence $\lk u\cap\lk t=C(p,w,q,c)\star\p(ab)=\lk ut$, and we can contract the edge $ut$.

Let $n\geq 4$, and suppose $C=C(t,a,b,\dots,d,t)$. We prove that if there is a vertex, say $z$, in $C$ that is not incident to a missing edge $e$ in $\lk u$ that is present in $\D$, then the pair $(u,z)$ satisfies the link condition.

Let $z$ be a vertex in $C$ that is not incident to $e$. If $z$ is not adjacent to $w$, then $\lk u\cap\lk z=\lk uz=\p(pqc)\star\p(xy)$, where $x$ and $y$ are the vertices in $C$ adjacent to $z$. Now, assume that $zw$ is an edge in $\D$. Let $\t^2=abc$, where $c$ is the central vertex of $\t^2$. Without loss of generality, let $z=a$. Then $\lk u\cap\lk z=\lk uz= \p(pqct)\#\p(pqcw)\#\p(pqwb)$ (since $wt,bt,pqc,pqw\notin\lk u$). Let $\t^2=pqt$, where $t$ is the central vertex of $\t^2$. If $t$ is not incident to any missing edge $e$ of $\lk u$ that is present in $\D$, then $pqt$ is a missing triangle in $\D$. Therefore, $\lk u\cap\lk t=C(c,p,w,q)\star\p(ad)=\lk ut$ (since $wc,ad\notin\lk u$, and $pq\notin\lk t$). On the other hand, if $a$ is not incident to $e$, then $\lk a\cap\lk u=\lk au=\p(bpqc)\#\p(pqct)\#\p(pqtw)$ (since $pqt, pqc\notin\lk u$). The situations for all the remaining vertices are trivial. Thus, in any situation, the pair $(u,z)$ satisfies the link condition.

Suppose every vertex of $C$ is incident to at least one missing edge in $\lk u$ that is present in $\D$. Then, $n=4$, $d(x)\geq 7$, and hence $g_2(\lk x)\geq 2$  for every vertex $x\in V(C)\cup V(pqc)$. However, $g_2(\lk x)\geq 2$ for every vertex $x\in V(C)\cup V(p,q,c)$ together with $g_3(\D)=0$ contradicts Lemma \ref{g-relations}. Hence, there is a vertex $z$ in $C$, which is not incident to a missing edge $e$ in $\lk u$ that is present in $\D$. 
 
 Thus, in every situation, there is a vertex, say $x$, in $\D$ such that the pair $(u, x)$ satisfies the link condition. Let $\D'$ be the resulting complex obtained after contracting the edge $ux$. Then $g_2(\D')\leq 3$, and $\D$ is obtained from $\D'$ by an edge expansion. This completes the proof.
\end{proof}
\begin{Lemma}\label{d=4, g2(lku=1), d(u)=6}
Let $\D\in\H_{4}$, and let $u$ be a vertex in $\D$ with $g_2(\lk u)=1$. Then $\D$ is obtained from a triangulated sphere with $g_2\leq 3$ by an edge expansion. 
\end{Lemma}
\begin{proof}
If $\D$ contains a vertex, say $x$, with $g_2(\lk x)=2$, then the result follows from Lemmas \ref{d=4, g2(lku=2), d(u)>=8, no missing edge} and \ref{d=4, g2(lku=2), d(u)>=8, one missing edge}. Now, assume that $g_2(\lk v)=1$ for every vertex $v\in\D$. 

Since $g_2(\lk u)=1$, it follows from Lemma \ref{d=4, g2(lku)=1, d(u)>=8} that $d(u)=6$. Let $\lk u=\p(abt)\star\p(pqc)$, where $abt$ and $pqc$ are two 2-simplices. Suppose $pqc$ is a face of $\D$. Then, $V(\lk pqc)$ is a subset of $V(\D\setminus\st u)$. Therefore, the degree of each vertex of $pqc$ in $\D$ is at least 9, which contradicts the fact of Lemma \ref{d=4, g2(lku)=1, d(u)>=8}. Hence, $pqc$ is a missing triangle in $\D$. Similarly, $abt$ is a missing triangle in $\D$. Therefore, the result follows from Lemma \ref{missing pqc}.
\end{proof}

\begin{Theorem}\label{g2=4. prime theorem}
Let $\D$ be a prime homology $4$-manifold with $g_2(\D)=4$. Then $\D$ is obtained from a triangulated $4$-sphere with $g_2\leq 3$ by one of the following operations:
\begin{enumerate}[$(i)$]
\item a bistellar $1$-move and an edge contraction,
\item an edge expansion.
\end{enumerate} 
\end{Theorem}
\begin{proof}
If $\D\notin\H_4$, then the result follows from Lemmas \ref{missing pqc},  \ref{missing edge}, and \ref{non prime vertex links}. Now, assume that $\D\in \H_4$. It follows from Lemma \ref{g2<3 exists} that, there exist a vertex, say $v$, in $\D$ such that $g_2(\lk v)\leq 2$. 

Let $u$ be a vertex in $\D$ with $g_2(\lk u)=2$. If $\lk u$ is an octahedral 3-sphere, then from Lemma \ref{octahedral 4-sphere} we get that $\D$ is obtained from a triangulated sphere with $g_2\leq 3$ by an edge expansion. However, if $\lk u$ is not an octahedral 3-sphere, then we can obtain the conclusion from Lemmas \ref{d=4, g2(lku=2), d(u)>=8, no missing edge} and \ref{d=4, g2(lku=2), d(u)>=8, one missing edge}. Likewise, if $\D$ contains a vertex $u$ with $g_2(\lk u)=1$, then from Lemma \ref{d=4, g2(lku)=1, d(u)>=8} we have $d(u)=6$, and hence the conclusion can be drawn from Lemma \ref{d=4, g2(lku=1), d(u)=6}.
\end{proof}

\section{4-dimensional Homology Manifolds with $g_2=5$}
\begin{Lemma}\label{g2=5, g2(lku=1),d(u)=9}
Let $\D\in\H_{5}$. If $u$ is a vertex in $\D$ with $g_2(\lk u)=1$, then $d( u)\leq 8$.
\end{Lemma}
\begin{proof}
If possible, let $d(u)\geq 9$. Since $g_2(\lk u)=1$, we have $\lk u=C\star\p(pqc)$, where $pqc$ is a 2-simplex, and $C$ is a cycle of length $n\geq 6$. Here, $pqc\in\D$, and each vertex of $C$ is incident to at least one diagonal edge in $\D[V(C)]$.

Suppose there are exactly three missing edges, $e_1$, $e_2$, and $e_3$, in $\lk u$, present in $\D$. Then $n=6$, and there are only two possible positions for the edges $e_1$, $e_2$, and $e_3$, as shown in Figure \ref{fig:3}. Note that no vertex of $C$ lies in $\lk pqc$, as it would make the link of that vertex not prime. Therefore, the triangle $pqc$ can never be in the link of the edges $e_1$, $e_2$, and $e_3$.

In the case of Figure \ref{fig:3} $(i)$, the link of the edge $sd$ contains no vertex of $C$, so $\lk sd$ contains at least two vertices from $\D\setminus\st u$. Furthermore, links of the other two diagonal edges in Figure \ref{fig:3} $(i)$ contain at most one vertex from $C$. Since $pqc\notin \lk e_i$, links of those two edges must contain at least one vertex from $\D\setminus\st u$. Similarly, the link of every diagonal edge in Figure \ref{fig:3}$(ii)$ contains at least two vertices from $\D\setminus\st u$.
\begin{figure}[ht]
\tikzstyle{ver}=[]
\tikzstyle{vertex}=[circle, draw, fill=black!100, inner sep=0pt, minimum width=4pt]
\tikzstyle{edge} = [draw,thick,-]
\centering
\begin{tikzpicture}[scale=0.75]
\begin{scope}[shift={(0,0)}]
\foreach \x/\y/\z in {1.6/-.5/d,0.3/.35/t,0.3/1.6/r,2.9/1.6/a,2.9/.35/b,1.6/2.45/s, 4/-0.3/q,7/-0.3/c,5.5/2.3/p}{\node[vertex] (\z) at (\x,\y){};}
\foreach \x/\y/\z in {1.6/-.9/d,1.6/2.9/s,0/.35/t,0/1.5/r,3.3/1.5/a,3.2/.35/b, 3.7/-.55/q,7.25/-.5/c,5.5/2.8/p,3/-1.8/(i)}{
\node[ver] () at (\x,\y){$\z$};}

\foreach \x/\y in {a/b,d/b,a/s,s/r,r/t,t/d,p/q,p/c,c/q}{\path[edge] (\x) -- (\y);}
\foreach \x/\y in {t/b,s/d,a/r}{\draw[densely dotted] (\x) -- (\y);}


\end{scope}
\begin{scope}[shift={(9.5,0)}]
\foreach \x/\y/\z in {1.6/-.5/d,0.3/.35/t,0.3/1.6/r,2.9/1.6/a,2.9/.35/b,1.6/2.45/s, 4/-0.3/q,7/-0.3/c,5.5/2.3/p}{\node[vertex] (\z) at (\x,\y){};}
\foreach \x/\y/\z in {1.6/-.9/d,1.6/2.9/s,0/.35/t,0/1.5/r,3.3/1.5/a,3.2/.35/b, 3.7/-.55/q,7.25/-.5/c,5.5/2.8/p,3/-1.8/(ii)}{
\node[ver] () at (\x,\y){$\z$};}

\foreach \x/\y in {a/b,d/b,a/s,s/r,r/t,t/d,p/q,p/c,c/q}{\path[edge] (\x) -- (\y);}
\foreach \x/\y in {r/b,s/d,a/t}{\draw[densely dotted] (\x) -- (\y);}


\end{scope}
\end{tikzpicture}
\caption{} \label{fig:3}
\end{figure}  

Thus, the number of edges between the vertices in $V(C)$ and $V(\D\setminus\st u)$ is at least 8. Since no vertices of $C$ lie in $\lk pqc$, we have $V(\lk pqc)\subseteq V(\D\setminus\st u)$. Therefore, the number of edges between $V(\D\setminus\st u)$ and $V(pqc)$ is at least 9. By Proposition \ref{complete graph}, $G(\D\setminus\st u)$ is a complete graph. Let $f_0(\D\setminus\st u)=m$. Then the total number of edges in $\D$ is at least ${m\choose 2}+8+9+f_1(\st u)+|\{e_1,e_2,e_3\}|={m\choose 2}+56$. Therefore, $g_2(\D)\geq \frac{m(m-1)}{2}+56-5(m+10)+15$, which simplifies to $m^{2}-11m+32\leq 0$. However, there is no positive integer $m$ for which the last inequality is satisfied. Therefore, the situations in Figure \ref{fig:3} are not possible. 

If $\lk u$ contains four missing edges that are present in $\D$, then $V(\D)=V(\st u)$. Consequently, $pqc$ is a missing triangle in $\D$, contradicting the fact that $\D\in \H_{5}$. Therefore, the result follows.
\end{proof}
\begin{Lemma}\label{g2=5, g2(lku=1),d(u)=8}
Let $\D\in\H_{5}$, and let $u$ be a vertex in $\D$ with $g_2(\lk u)=1$. If $d(u)=8$, then $\D$ is obtained from a triangulated $4$-sphere with $g_2\leq 4$ by an edge expansion, together with flipping an edge or bistellar $2$-moves (once or twice).
\end{Lemma}
\begin{proof}
Since $g_2(\lk u)=1$ and $d(u)=8$, we have $\lk u=C\star\p(pqc)$, where $pqc$ is a 2-simplex, and $C$ is a cycle of length $5$. Here, $pqc\in\D$, and each vertex of $C$ is incident to at least one diagonal edge in $\D[V(C)]$. Thus, the number of missing edges in $\lk u$ that are present in $\D$ is at least three.

\begin{figure}[ht]
\tikzstyle{ver}=[]
\tikzstyle{vertex}=[circle, draw, fill=black!100, inner sep=0pt, minimum width=4pt]
\tikzstyle{edge} = [draw,thick,-]
\centering
\begin{tikzpicture}[scale=0.75]
\begin{scope}[shift={(0,0)}]
\foreach \x/\y/\z in {0.3/0/t,0.3/1.6/r,2.9/1.6/a,2.9/0/b,1.6/2.4/s, 4/0/q,7/0/c,5.5/2.3/p}{\node[vertex] (\z) at (\x,\y){};}
\foreach \x/\y/\z in {1.6/2.9/s,0/-.3/t,0/1.5/r,3.3/1.5/a,3.2/-.3/b, 3.8/-.3/q,7.2/-.3/c,5.5/2.8/p}{
\node[ver] () at (\x,\y){$\z$};}

\foreach \x/\y in {a/b,a/s,b/t,s/r,r/t,p/q,p/c,c/q}{\path[edge] (\x) -- (\y);}
\foreach \x/\y in {s/b,a/r,a/t}{\draw[densely dotted] (\x) -- (\y);}


\end{scope}
%
%
%
\end{tikzpicture}
\caption{} \label{fig:4}
\end{figure}

Suppose there are exactly three missing edges, denoted as $e_1, e_2,$ and $e_3$, in $\lk u$ that are present in $\D$. The only possible arrangement for these edges in $C$ is illustrated in Figure \ref{fig:4} with dotted lines. Let us consider the vertex set of the cycle $C$ as $\{a, b, t, r, s\}$.
Note that $pqc\notin\lk e_i$ for every $i$. Therefore, $\lk sb$ contains at least one vertex from $\D\setminus\st u$. If $\lk ar$ does not contain any vertex from $\D\setminus\st u$, then $\lk ar=\p(st)\star\p(pqc)$ (since $st\notin\D$ and $pqc\notin\lk a$). Let us apply the operation of flipping the edge $ar$ with $st$ (replace $ar\star\p(st)\star\p(pqc)$ with $\p(ar)\star st\star\p(pqc)$) in $\D$ and denote the resulting complex by $\D'$. Then $\D'$ is a homology manifold with $g_2=5$, and $\lk u=lk_{\D'}u$. Although $lk_{\D'}u$ contains three missing edges that are present in $\D'$, the vertex $r\in C$ is not incident to any of the missing edges in $\D'$. Hence, $lk_{\D'}u\cap lk_{\D'}r=lk_{\D'}ur=\p(pqc)\star\p(st)$, and we can contract the edge $ur$. 
Let $\D''$ be the resulting complex. Then $g_2(\D'')=4$, and $\D'$ is obtained from $\D''$ by an edge expansion. Furthermore, $\D$ is obtained from $\D'$ by flipping an edge. Similarly, if $\lk at$ does not contain any vertex from $\D\setminus\st u$, we can perform the operation of flipping the edge $at$ with $br$ in $\D$. This ensures that the pair $(a,t)$ satisfies the link condition in this case. Therefore, $\D$ is obtained by flipping an edge together with an edge expansion from a triangulated 4-sphere with $g_2=4$.

It remains to consider the case where $\lk e_i$ contains at least one vertex from $\D\setminus\st u$ for each $i$. In this scenario, $d(x)\geq 8$ for every vertex $x\in \lk u$, and specifically, $d(x)\geq 9$ for each vertex $x$ in the set $\{a,p,q,c\}$.

Let the number of vertices in $\D\setminus\st u$ be $m$. By Proposition \ref{complete graph}, $G(\D\setminus\st u)$ forms a complete graph. Given that $\lk e_i$ contains at least one vertex from $\D\setminus\st u$, the number of edges between vertices of $\D\setminus\st u$ and vertices of $C$ is at least 5. Since $V(\lk pqc)\subseteq V(\D\setminus\st u)$, the number of edges between the vertices of $pqc$ and the vertices in $\D\setminus\st u$ is at least 9, with equality occurring when $\lk pqc$ is a 3-cycle. Therefore, $f_1(\D)\geq {m\choose 2}+5+9+f_1(\st u)+|\{e_1,e_2,e_3\}|={m\choose 2} +48$. Additionally, if $V(\D)\setminus V(\st u\cup\lk pqc)\neq\emptyset$, then there will be more edges, as we will see when necessary.

We first claim that $\lk pqc$ is a 3-cycle. If $\lk pqc$ is a cycle of length greater than 3, then $f_1(\D)\geq {m\choose 2} +48+3k$ for some natural number $k$. However, $g_2(\D)=5$ implies that $5\geq {m\choose 2} +48+3k-5(9+m)+15$. Simplifying this equation yields $m^2 - 11m + 26 + 6k\leq 0$, where $k\geq 1$. However, there is no natural number $m$ satisfying this inequality. Therefore, $\lk pqc$ is indeed a 3-cycle. Let $\lk pqc=\p(xyz)$ for some vertices $x, y,$ and $z$ in  $\D\setminus\st u$.

We calculate the possible values of $m$.  Note that $g_2(\D)=5$, and $f_1(\D)\geq {m\choose 2} + 48$, which implies $m^2 - 11m + 26 \leq 0$, where $m\geq 3$. The only positive integers satisfying these inequalities are $m=4, 5, 6$, and $7$. 

Let $m=7$. Since $G(\D\setminus\st u)$ is a complete graph and $\lk pqc=\p(xyz)$, the degree of each of the vertices $x, y$, and $z$ is at least 9. Thus, $\D$ contains at least seven vertices with a degree of at least 9.
Now, with $f_0(\D)=16$ and the relation $\sum_{v\in\D}g_2(\lk v)\leq 20$, there is at least one vertex, denoted as $x$, with $d(x)\geq 9$ and $g_2(\lk x)=1$. This contradicts the fact of Lemma \ref{g2=5, g2(lku=1),d(u)=9}. Hence $m\leq 6$.

Let $m\leq 6$. If $xyz\notin\D$, then we apply a bistellar 2-move with respect to the pair $(pqc,xyz)$. Let $\D'$ be the resulting complex. According to Lemma \ref{missing pqc}, $\D'$ is obtained by an edge expansion from a triangulated sphere with $g_2=2$, and $\D$ is obtained from $\D'$ by a bistellar 2-move. If $xyz\in\D$, then $\lk xyz$ forms a cycle. Furthermore, $\lk xyz$ contains no vertices from $\{p,q,c\}$. If $\lk xyz$ contains a vertex from $C$, then there are at least 9 edges between the vertices of $C$ and $V(\D\setminus\st u)$, which implies $f_1(\D)\geq {m\choose 2}+11+9+f_1(\st u)+|\{e_1,e_2,e_3\}|={m\choose 2}+52$. However, this contradicts $g_2=5$. Thus, $\lk xyz$ is a 3-cycle, denoted as $\p(x_1x_2x_3)$, where $\{x_1,x_2,x_3\}\subseteq V(\D\setminus\st u)$. Clearly, $x_1x_2x_3$ is a missing triangle in $\D$, as otherwise, $\lk x_1x_2x_3$ would be a cycle containing vertices from $\lk u$, contradicting $g_2(\D)=5$. Now, we apply bistellar 2-moves with respect to pairs $(x_1x_2x_3,xyz)$ and $(pqc,xyz)$. Let $\D''$ be the resulting complex. Using similar arguments, $\D$ is obtained by two repeated bistellar 2-moves and an edge expansion from a triangulated sphere with $g_2\leq 2$.
 
If $\lk u$ contains four missing edges present in $\D$, then $V(\D)=V(\st u)$. This implies that $pqc$ is a missing triangle in $\D$, which leads to a contradiction.  This completes the proof.
\end{proof}
\begin{Lemma}\label{g2=5, g2(lku=1),d(u)=7}
Let $\D\in\H_{5}$, and let $u$ be a vertex in $\D$ with $g_2(\lk u)=1$. If $d(u)=7$, then $\D$ is obtained from a triangulated sphere with $g_2\leq 4$ by an edge expansion together with flipping an edge (possibly zero times) or a bistellar $2$-move (possibly zero times).
\end{Lemma}
\begin{proof}
Let $\lk u=C(a,b,d,t)\star\p(pqc)$, where $pqc$ is a 2-simplex. Here, $pqc\in\D$, and every vertex of $C$ is incident to at least one diagonal edge in $\D[V(C)]$. Thus, $\D$ has exactly two diagonal edges with endpoints in $C$. Let $G(\D[V(\lk u)])=G(\lk u)\cup\{e_1,e_2\}$, where $e_1=ad$ and $e_2=bt$. Note that $V(\lk pqc)\subseteq V(\D\setminus \st u)$. By a similar argument as in Lemma \ref{d=4, g2(lku)=1, d(u)>=8}, we find that $\lk ad$ and $\lk bt$ each contain at least one vertex from $V(\D\setminus\st u)$. Hence, $d(x)\geq 8$ for every vertex $x\in \lk u$. Thus, $\D$ has at least seven vertices (the vertices of $\lk u$) with a degree of at least 8. If any of these vertex-links has $g_2=1$, the result follows from Lemma \ref{g2=5, g2(lku=1),d(u)=8}. For the remaining case, where $g_2(\lk x)\geq 2$ for every vertex $x\in\lk u$, using the relations $\sum_{v\in\D} g_2(\lk v)=3g_3(\D)+4g_2(\D)$ and $g_3(\D)=0$, we get $f_0(\D\setminus\st u)\leq 5$.

Let $\lk pqc$ be a 3-cycle, denoted as $\p(\tau)$, where $\tau$ is a 2-simplex. If $\tau\in\D$, then $\lk \tau$ is a cycle. Since no vertex from $\{p,q,c\}$ is in $\lk \tau$, the degree of each vertex in $\tau$ is at least 8. It follows that $\D$ contains at least one vertex, say $x$, with $d(x)\geq 8$ and $g_2(\lk x)=1$,  which leads to the conclusion as in Lemma \ref{g2=5, g2(lku=1),d(u)=8}. Now, assume that $\tau\notin\D$, and replace the 4-dimensional ball $pqc\star\p(\tau)$ with $\tau\star\p(pqc)$. In the resulting complex $\D'$, $lk_{\D'} u$ is the same as $\lk u$, but $pqc$ is a missing triangle here. Therefore, by Lemma \ref{missing pqc}, $\D$ is obtained from an edge expansion on a triangulated sphere with $g_2=3$ by a bistellar 2-move.

Now we assume that $\lk pqc$ is a cycle of length greater than 3. Let $f_0(\D\setminus\st u)=4$ and $\lk pqc=C(x_1,x_2,x_3,x_4)$.  We claim that there is at least one vertex, say $x_i$, with $d(x_i)=6$. 

If possible, let $d(x_i)\geq 7$ for $1\leq i \leq 4$. Without loss of generality, let $x_1\in\lk ad$. If $x_1x_3\in \D$, then $d(x_1)\geq 8$, and we have $f_1(\D)\geq {8\choose 2} + 3\cdot 7+8 - {4\choose 2} = 51$. If $x_1x_3\notin \D$, then $f_1(\D)\geq {8\choose 2} + 4\cdot 7 - ({4\choose 2}-1) = 51$.  This contradicts $g_2(\D)=5$. Thus, there is a vertex in $\lk pqc$ whose link contains exactly six vertices. 

Without loss of generality, let $x_2$ be a vertex with $d(x_2)=6$. Then $g_2(\lk x_2)=1$ and $x_2$ is adjacent to at least two vertices in the cycle $\lk pqc$. Let $\lk x_2=\p(\tau_1)\star\p(\tau_2)$, where $\tau_1$ and $\tau_2$ are $2$-simplices. If both 2-simplices are present in $\D$, then the link of each vertex in $\lk pqc$ adjacent to $x_2$ has a degree of at least 8. Hence, $\D$ has at least one vertex, say $x$, with $g_2(\lk x) = 1$ and $d(x) = 8$, leading to the conclusions as in Lemma \ref{g2=5, g2(lku=1),d(u)=8}. If either $\tau_1$ or $\tau_2$ is a missing simplex in $\D$, then the link condition is satisfied for a pair $(u,x)$, where $x$ is any vertex in that triangle. So, $\D$ is obtained by an edge expansion from a triangulated sphere with $g_2=4$. 

 Let $f_0(\D\setminus\st u)=5$, $V(\D\setminus\st u)=\{x_1,x_2,x_3,x_4,x_5\}$, and $\lk pqc=C(x_1,x_2,x_3,x_4)$.  In this case, $g_2(\lk x_j)=1$ for each $j$. We claim that there exists an $i$ such that $d(x_i)$ is either $6$ or at least 8. Suppose, for the sake of contradiction, that $d(x_i)=7$ for $1\leq j \leq 5$. Since $g_2(\D)=5$, at most two edges are missing, and their endpoints are in $\{x_1,x_2,x_3,x_4,x_5\}$.  Let $x_1 \in \lk ad$ and $x_k \in \lk bt$. If $x_1 = x_k$, then $d(x_1) \geq 9$, leading to a contradiction with Lemma \ref{g2=5, g2(lku=1),d(u)=9}. Now, consider the case where $k \in \{2, 3, 4\}$. In this scenario, $x_1x_3$, $x_1x_5$, and $x_kx_5$ are missing edges in $\D$, leading to a contradiction. Therefore, $x_5 \in \lk bt$. Since $x_1x_3$ and $x_1x_5$ are missing edges, all the remaining vertices are adjacent to each other. Given that $d(x_i) = 7$ for $1 \leq j \leq 5$, it follows that $x_2$ and $x_4$ are not adjacent to any vertices in $\{a, b, d, t\}$, $x_3$ is adjacent to exactly one vertex in $\{a, b, d, t\}$, and $x_5$ is adjacent to two vertices in $\{a, d, p, q, c\}$. Thus, we cannot have five vertices in both the links $\lk bx_5$ and $\lk tx_5$, which leads to a contradiction.  Consequently, neither $\lk ad$ nor $\lk bt$ contains any vertices from the set $\{x_1, x_2, x_3, x_4\}$, while $x_5$ is present in both $\lk ad$ and $\lk bt$. Since $\lk bx_5$ contains at least five vertices, we must have at least two edges either from $b$ to a vertex in the set  $\{x_1, x_2, x_3, x_4\}$ or from $x_5$ to a vertex in the set $\{p, q, c\}$. However, in both of these cases, we would end up with more than two missing edges, with both endpoints in the set $\{x_1, x_2, x_3, x_4, x_5\}$, leading to a contradiction. Therefore, the claim is established, and the result follows from similar arguments as in the previous paragraph.

 Let $f_0(\D\setminus\st u)=5$ and $\lk pqc=C(x_1,x_2,x_3,x_4, x_5)$.  In this case, $g_2(\lk x_j)=1$ for each $j$. We claim that there exists an $i$ such that $d(x_i)$ is either $6$ or at least 8. If $d(x_i)=7$ for $1\leq j \leq 5$, then at most two edges are missing, and their endpoints are in $\{x_1,x_2,x_3,x_4,x_5\}$. Let $x_1\in\lk ad$ and $x_k\in\lk bt$. If $x_1=x_k$, then $d(x_1)\geq 9$, leading to a contradiction. Now, let $k\in \{2,3,4,5\}$. At least three edges  from the set $\{x_1x_3,x_1x_4$, $x_kx_{k-2},x_kx_{k+2}\}$ are missing edges in $\D$, which again results in a contradiction. Therefore, the claim is established. Now the result follows from arguments similar to those in the fifth paragraph of this lemma. 
\end{proof}

\begin{Lemma}\label{g2=5, g2(lku=2), d(u)>=8, no missing edge}
Let $\D$ be in $\H_{5}$, and let $u$ be a vertex in $\D$ with $g_2(\lk u)=2$. If $d(u)\geq 8$ and $\lk u$ contains at most one missing edge that is present in $\D$, then $\D$ is obtained from a triangulated sphere with $g_2\leq 4$ through an edge expansion.
\end{Lemma}
\begin{proof} 
On one hand, if $\lk u$ is an octahedral 3-sphere, then the result follows from Lemma \ref{octahedral 4-sphere}. On the other hand, if $G(\D[V(\lk u)])=G(\lk u)$, then the result follows from arguments similar to those in Lemma \ref{d=4, g2(lku=2), d(u)>=8, no missing edge}. Now, suppose $\lk u$ is not an octahedral 3-sphere, and $G(\D[V(\lk u)])=G(\lk u)\cup \{e\}$. Let $K, K', \t^2$, and $w$ be as described in Lemma \ref{d=4, g2(lku=2), d(u)>=8, no missing edge}.

Suppose $K$ is not prime, and let $lk_{K} \t^2=\{r,s\}$, where $s$ is the vertex added to $K'$ in the last facet subdivision. If $e$ is not incident to $w$, then $\lk u\cap\lk w=\p(\t^2)\star\p(rs)=\lk uw$. On the other hand, if $e$ is not incident to $s$, then $\lk u\cap\lk s=\p(\t^2)\cap\p(wx)=\lk us$, where $lk_{K} s=\p(x\star\t^2)$. Thus, in both cases, we can contract an edge. Let $\D'$ be the resulting complex after contracting the edge. Then $g_2(\D')=4$, and $\D$ is obtained from $\D'$ by an edge expansion.

If $K$ is prime, then $K=K'=C\star\p(pqc)$, where $C$ is a cycle of length $n\geq 4$. Let $y$ be a vertex in $C$ that is not incident to $e$. By the same argument as in Lemma \ref{d=4, g2(lku=2), d(u)>=8, one missing edge}, the pair $(u,y)$ satisfies the link condition, and we are done.     
\end{proof}
\begin{Lemma}\label{g2=5, g2(lku=2), two missing edge}
Let $\D\in\H_{5}$, and let $u$ be a vertex in $\D$ with $g_2(\lk u)=2$. If $d(u)\geq 8$ and $\lk u$ contains exactly two missing edges that are present in $\D$, then $\D$ is obtained from a triangulated sphere with $g_2\leq 4$ by an edge expansion together with flipping an edge (possibly zero times) or a bistellar $2$-move (possibly zero times).
\end{Lemma}
\begin{proof}
If $\lk u$ is an octahedral 3-sphere, then the result follows from Lemma \ref{octahedral 4-sphere}. Now, assume that $\lk u$ is not an octahedral 3-sphere, and let $G(\D[V(\lk u)])=G(\lk u)\cup \{e_1,e_2\}$. Let $K,K',\t^2$, and $w$ be as in Lemma \ref{d=4, g2(lku=2), d(u)>=8, no missing edge}.
\vst
\noindent\textbf{Case 1:} Let $K$ not be prime. Let $s_1,\dots,s_m$ be the vertices added to $K'$ during the subdivision of tetrahedra to form $K$. By applying similar arguments as in Lemmas \ref{d=4, g2(lku=2), d(u)>=8, no missing edge} and \ref{d=4, g2(lku=2), d(u)>=8, one missing edge}, we can show that if at least one vertex in the set $\{w,s_1,\dots,s_m\}$ is not incident to some $e_i$, then $\D$ is obtained by an edge expansion from a triangulated sphere with $g_2\leq 4$. For the remaining part of the proof, we assume that each vertex in the set $\{w,s_1,\dots,s_m\}$ is incident to at least one edge in $\{e_1,e_2\}$. Therefore, $m\leq 3$.

Let $K'=\p(abt)\star\p(pqc)$, where $abt$ is a 2-simplex. Without loss of generality, let $\t^2=pqt$. Clearly, $lk_{K'} (pqt)=\{a,b\}$. Suppose $m=1$, and let $pqta$ be the missing tetrahedron in $K$ such that $lk_{K}s_1=\p(pqta)$. Thus, the degree of each vertex of $\lk u$ is at least 7. If $g_2(\lk x)=1$ for a vertex $x$ in $\lk u$, then the result follows from Lemmas \ref{g2=5, g2(lku=1),d(u)=9}, \ref{g2=5, g2(lku=1),d(u)=8}, and \ref{g2=5, g2(lku=1),d(u)=7}. Let $g_2(\lk x)\geq 2$ for every vertex in $\lk u$. Then, the relation $\sum_{v\in\D}g_2(\lk v)=4g_2(\D)$ implies that $\sum_{v\in\D\setminus\st u}g_2(\lk v)\leq 2$. Therefore, $\D\setminus\st u$ contains at most two vertices. 

In $\lk u$, the missing edges are $wc,wa,s_1b,$ and $s_1c$. If $abt$ is a missing triangle in $\D$, then $\lk u\cap\lk b=\p(wpqt)\#\p(pqtc)\#\p(pqca)=\lk ub$, since $aw,wc,pqc,pqt\notin\lk u$, and $at\notin\lk b$. On the other hand, if $pqc\notin \D$, then $\lk pu=(C(a,q,w,t)\star\p(bs_1))\#_{abt}\p(abtc)$. We claim that $\lk u\cap\lk p=\lk pu$. Notice that $abt\notin\lk u$ and $qc\notin\lk p$. Therefore, to establish the claim, it is sufficient to prove that $pqt\notin\D$. If $pqt\in\D$, then $\lk pqt$ is a cycle, which does not contain $c$ (since $pqc\notin\D$). Furthermore, no vertex from the set $\{a,b,s_1,w\}$ is in $\lk pqt$; otherwise $\p(upqt)$ would be a missing tetrahedron in the corresponding vertex link. Thus,  $V(\lk pqt)\subseteq V(\D\setminus\st u)$, which is not possible. Therefore, $pqt\notin\D$, and we conclude that $\lk u\cap\lk p=\lk pu=(C(a,q,w,t)\star\p(bs_1))\#_{abt}\p(abtc)$. 

Let us assume that both $abt$ and $pqc$ are faces of $\D$. Suppose both $wc$ and $s_1c$ are present in $\D$, i.e., $\{e_1,e_2\}=\{wc,s_1c\}$. Then $abt$ is a missing triangle in $\D$; because if $abt\in\D$, then $\lk abt\subseteq\D\setminus \st u$, and hence $\D\setminus \st u$ contains at least three vertices, which is a contradiction. If at least one of $wc$ and $s_1c$ is a missing edge in $\D$, then one of the links $\lk abt$ and $\lk pqc$ contains at least two vertices, and the other contains at least one vertex from $V(\D\setminus\st u)$. Thus,  $V(\D\setminus\st u)$ contains exactly two vertices, say $x_1$ and $x_2$, and $g_2(\lk x_i)=1$ for each $i$. By Proposition \ref{complete graph}, the edge $x_1x_2\in\D$. Therefore, at least one of $x_1$ and $x_2$ is adjacent to more than six vertices, i.e., there is a vertex $x\in\D$ with $g_2(\lk x)=1$ and $d(x)\geq 7$. Hence the result follows from Lemmas \ref{g2=5, g2(lku=1),d(u)=9}, \ref{g2=5, g2(lku=1),d(u)=8}, and \ref{g2=5, g2(lku=1),d(u)=7}.

 Let $m\geq 2$. If the pair $(u,c)$ satisfies the link condition, then we are done. Suppose that $\lk u\cap\lk c\neq\p(abt)\star\p(pq)=\lk uc$. Then $c$ must be incident to at least one more vertex other than the vertices in $\{a,b,t,p,q,u\}$. Therefore, the degree of each vertex in $\lk u$ is at least 7. If $g_2(\lk x)\geq 2$ for every vertex in $\lk u$, then $f_0(\D)\leq 10$. Thus,  there is no vertex in $\D\setminus\st u$, which contradicts $g_2(\D)=5$. Therefore, $g_2(\lk x)=1$ for a vertex $x\in\lk u$, and hence the result.
 
 Let $C$ be a cycle of length $n\geq 4$. Since $\lk u$ has at most two missing edges that are present in $\D$, and $\{w,s_1,\dots,s_m\}\subseteq V(\{e_1,e_2\})$, there is at least one vertex in the cycle $C$ that is not incident to the edges $e_i$, $1\leq i\leq 2$. Hence, the result follows from similar arguments as in Lemma \ref{d=4, g2(lku=2), d(u)>=8, one missing edge}.
 
\vst
\noindent\textbf{Case 2:} Let $K$ be prime. Then $K=C\star\p(pqc)$, where $C$ is a cycle of length $n\geq 4$. If there is a vertex $z$ in $C$ that is not incident to a missing edge $e_i$, then by the same arguments as in Lemma \ref{d=4, g2(lku=2), d(u)>=8, one missing edge}, $(u,z)$ satisfies the link condition. Now, suppose each vertex of $C$ is incident to at least one missing edge. Then $n=4$ and $w$ is not incident to the edges $e_i$, $1\leq i\leq 2$. Let $C=C(a,b,d,t)$. If the central vertex of $\t^2$ belongs to $C$, then the pair $(u,w)$ satisfies the link condition trivially. 

Let the central vertex of $\t^2$ belong to $pqc$. Without loss of generality, let $\t^2=abc$, where $c$ is the central vertex of $\t^2$. Then $\lk uw=\p(abc)\star\p(pq)$. If $wpq\notin\D$, then the pair $(u,w)$ satisfies the link condition. Let $wpq\in\D$, where $\t^2=abc$. Then $\lk wpq$ contains at least three vertices. 
Note that $ac\notin\text{lk}wpq$; otherwise, $a\in\text{lk}pqc$, and hence $\partial(upqc)\subseteq \text{lk}a$. This contradicts the fact that $\text{lk}a $ is prime. Since $wpq\notin\lk u$, it follows that $\lk wpq$ contains at least one more vertex other than $a$ and $b$. Therefore, the degree of $w$ is at least 7. Thus, $d(x)\geq 7$ for every $x\in\st u$. If $g_2(\lk x)=1$ for a vertex $x$ in $\lk u$, then the result follows from Lemmas \ref{g2=5, g2(lku=1),d(u)=9}, \ref{g2=5, g2(lku=1),d(u)=8}, and \ref{g2=5, g2(lku=1),d(u)=7}. Let $g_2(\lk x)\geq 2$ for every vertex $x\in\lk u$. Then Lemma \ref{g-relations} and $g_3(\D)=0$ imply that the number of vertices in $\D\setminus\st u$ is at most 2. 

If $pqc\notin\D$, then $\lk c\cap\lk u=\lk cu=C(a,w,b,d,t)\star\p(pq)$ holds. On the other hand, if $pqc\in\D$, then $\lk pqc$ contains at most one vertex from $\lk u$. Note that $\lk pqc$ cannot contain a vertex of $\lk u$ other than $w$. Therefore, $\D\setminus\st u$ contains exactly two vertices, and $\lk pqc=ww_1w_2$, where $\{w_1,w_2\}=V(\D\setminus\st u)$. Clearly, $g_2(\lk w_i)=1$, for each $i$. If $ww_1w_2\in\D$, then $\lk ww_1w_2$ contains at least three vertices from $C$, and hence $d(w_i)\geq 7$. If $ww_1w_2$ is a missing triangle, then we will replace the 4-dimensional ball $pqc\star\p(ww_1w_2)$ with $\p(pqc)\star ww_1w_2$. If $\D'$ is the resulting complex, then $g_2(\D')=5$. Furthermore, $\lk u=lk_{\D'}u$ and $pqc$ is a missing triangle in $lk_{\D'}u$. Thus,  $\D$ is obtained by an edge expansion and bistellar 2-move from a triangulated sphere with $g_2\leq 4$.

In both Case 1 and Case 2, whenever a link condition arises, we can apply an edge contraction. Thus, $\D$ is obtained from a triangulated sphere with $g_2\leq 4$ either by an edge expansion or by a combination of an edge expansion and a bistellar 2-move.  
\end{proof}
\begin{Lemma}\label{g2=5, g2(lku=2), three missing edge}
Let $\D\in\H_{5}$, and let $u$ be a vertex in $\D$ with $g_2(\lk u)=2$. If  $\lk u$ contains exactly three missing edges that are present in $\D$, then $\D$ is obtained from a triangulated sphere with $g_2\leq 4$ by an edge expansion.
\end{Lemma}
\begin{proof}
If $\lk u$ is an octahedral 3-sphere, then the result follows from Lemma \ref{octahedral 4-sphere}. Now, assume that $\lk u$ is not an octahedral $3$-sphere, and take $G(\D[V(\lk u)])=G(\lk u)\cup \{e_1,e_2,e_3\}$. By Proposition \ref{same vertex set}, $V(\D)=V(\st u)$. Let $K,K',\t^2$, and $w$ be as in Lemma \ref{d=4, g2(lku=2), d(u)>=8, no missing edge}.

Let $K$ not be prime, and let $s_1,\dots,s_m$ be the vertices added to $K'$ during facet subdivisions.  Furthermore, we assume that each vertex in $\{w,s_1,\dots, s_m\}$ is incident to at least one edge $e_i$; otherwise, the result follows from same arguments as in Lemmas \ref{d=4, g2(lku=2), d(u)>=8, no missing edge} and \ref{d=4, g2(lku=2), d(u)>=8, one missing edge}. Therefore $m\leq 5$. 

Suppose $C$ is a $3$-cycle, i.e., $n=3$, and consider $C=C(a,b,t)$. Without loss of generality, let $\tau^2=pqt$, and consider $pqta$ as a missing tetrahedron in $K$.

Our primary goal is to establish that $abt$ is a missing triangle in $\D$. Assuming, for the sake of argument, that $abt$ is a face of $\D$, we note that no vertex from the set $\{p, q, c\}$ can belong to $\lk abt$. Including any such vertex would contradict the fact that the link of each vertex in $\D$ is prime. Therefore, $V(\lk abt)\subseteq\{w,s_1,\dots,s_m\}$, indicating $m\geq 2$. Since $\t^2$ participates in every facet subdivision, each vertex in $\{w,s_1,\dots,s_m\}$ is incident to at most one vertex from the set $\{a,b\}$. Furthermore, $wa\notin\lk u$, and at most one $s_j$ is adjacent to $a$ (or $b$) in $\lk u$. Thus, the condition $V(\lk abt)\subseteq \{w,s_1,\dots,s_m\}$ implies that $\lk u$ has at least four missing edges that are present in $\D$. This is a contradiction. Therefore, we conclude that $abt$ is a missing triangle in $\D$.

Note that the vertex $t$ is not incident to any of the edges  $e_i$, $1\leq i\leq 3$. Considering the appropriate subdivisions and the fact that none of the vertices from the set $\{a,b,w,s_1,\dots,s_m\}$ are in $\lk pqt$,  it is clear that the pair $(u,t)$ satisfies the link condition.

Let $K'=C\star\p(pqc)$, where $C$ is a cycle of length $n\geq 4$. If there is a vertex $z$ in $C$ that is not incident to $e_j$, then by similar arguments as in Lemma \ref{d=4, g2(lku=2), d(u)>=8, one missing edge}, the pair $(u,z)$ satisfies the link condition. Now, assume that every vertex of $C$ is incident to some $e_j$. It follows that no vertex of $pqc$ is incident to $e_i$, and hence $pqc$ is a missing triangle in $\D$. Thus, there is a vertex, say $x$, in $\{p,q,c\}$ such that the pair $(u,x)$ satisfies the link condition. 

Let $K$ be prime and $K=C\star\p(pqc)$, where $C$ is a cycle of length $n$. Since the number of missing edges is 3, we have $n\geq 5$. If there is a vertex $z$ in $C$ that is not incident to any missing edge $e_i$, then by similar arguments as in Lemma \ref{d=4, g2(lku=2), d(u)>=8, one missing edge}, the pair $(u,z)$ satisfies the link condition. Let every vertex of $C$ be incident to some $e_i$. Then $n=5$. Since $p,q,$ and $c$ are adjacent to every vertex of $C$ in $\lk u$, it follows that $e_i$ is not incident $p,q$, or $c$. Therefore, $pqc$ is a missing triangle in $\D$, and there is a vertex, say $x$, in $\{p,q,c\}$ such that $(u,x)$ satisfies the link condition. Hence the proof.
\end{proof}
\begin{Lemma}\label{g2=5, g2(lku=2), d(u)=7}
Let $\D\in\H_{5}$, and let $u$ be a vertex in $\D$ with $g_2(\lk u)=2$. If $d(u)=7$, then $\D$ is obtained from a triangulated sphere with $g_2\leq 4$ by an edge expansion together with flipping an edge (possibly zero times) or a  bistellar $2$-move (possibly zero times).
\end{Lemma}
\begin{proof}
Let $K, K', \t^2$, and $w$ be as defined in Lemma \ref{d=4, g2(lku=2), d(u)>=8, no missing edge}. Since $d(u)=7$, $K$ is prime, and $V(K)=6$. Suppose $K=\p(abt)\star\p(pqc)$, where $abt$ is a 2-simplex. Without loss of generality, let $\t^2=pqt$. Then $\lk u$ has at most one missing edge. If any one of the triangles $pqc$ and $abt$ is a missing triangle in $\D$, then there exists a vertex, say $x$, within that triangle such that the pair $(u,x)$ satisfies the link condition.

Suppose both the triangles $abt$ and $pqc$ are faces of $\D$. Then the links of both triangles contain at least two vertices from $V(\D\setminus\st u)$. Therefore, the degree of each vertex in $\{p,q,t,a,b\}$ is at least 9, and $d(c)\geq 8$. If $g_2(\lk x)\geq 3$ for every vertex $x\in\{p,q,c,a,b,t\}$, then $\sum_{v\in\D}g_2(\lk v)\geq 23$. However, the relations $\sum_{v\in\D}g_2(\lk v)=3g_3(\D)+4g_2(\D)$ and $g_3(\D)=0$ imply that $\sum_{v\in\D}g_2(\lk v)=20$. Therefore, $\lk u$ contains a vertex, say $x$, with $d(x)\geq 8$ and $g_2(\lk x)\leq 2$. Hence, the result follows from Lemmas \ref{octahedral 4-sphere}, and \ref{g2=5, g2(lku=1),d(u)=9} - \ref{g2=5, g2(lku=2), three missing edge}. 
\end{proof}
\begin{Lemma}\label{g2=5, g2(lku=1), d(u)=6}
Let $\D\in\H_{5}$, and let $u$ be a vertex in $\D$ with $g_2(\lk u)=1$. If $d(u)=6$, then $\D$ can be obtained from a triangulated sphere with $g_2\leq 4$ by an edge expansion together with flipping an edge (possibly zero times) or a bistellar $2$-move (possibly zero times).
\end{Lemma}
\begin{proof}
Let $\lk u=\p(abt)\star\p(pqc)$, where $abt$ and $pqc$ are two 2-simplices. If either $abt$ or $pqc$ is a missing triangle in $\D$, then for every vertex $x$ in that triangle, the pair $(u,x)$ satisfies the link condition. Suppose both the triangles $abt$ and $pqc$ are present in $\D$. Then all the vertices of $\lk pqc$ and $\lk abt$ are in $V(\D\setminus\st u)$. Also, $d(x)\geq 9$ for every vertex $x\in\lk u$. If $g_2(\lk x)\geq 3$ for every vertex $x\in\lk u$, then $\sum_{v\in\D}g_2(\lk v)\geq \sum_{v\in\D\setminus\st u}g_2(\lk v)+3\cdot 6+1$. This implies that $\D\setminus\st u$ contains at most two vertices, which contradicts the fact that $V(\lk pqc)\subseteq V(\D\setminus\st u)$. Therefore, there is at least one vertex, say $x$, in $\lk u$ such that $d(x)\geq 9$ and $g_2(\lk x)= 2$. Thus, the result follows from Lemmas \ref {g2=5, g2(lku=2), d(u)>=8, no missing edge}, \ref{g2=5, g2(lku=2), two missing edge}, and \ref{g2=5, g2(lku=2), three missing edge}.
\end{proof}
\begin{Theorem}\label{g2=5. prime theorem}
Let $\D$ be a prime homology $4$-manifold with $g_2(\D)=5$. Then $\D$ is a triangulated sphere obtained from a triangulated $4$-sphere with $g_2\leq 4$ by one of the following operations:
\begin{enumerate}[$(i)$]
\item a bistellar $1$-move and an edge contraction,
\item an edge expansion together with flipping an edge (possibly zero times) or a bistellar $2$-move (possibly zero times).
\end{enumerate}
\end{Theorem}
\begin{proof}

If $\D\notin\H_5$, then the result follows from Lemmas \ref{missing pqc},  \ref{missing edge}, and \ref{non prime vertex links}. Now, assume that $\D\in \H_5$. It follows from Lemma \ref{g2<3 exists} that there exists a vertex, say $v$, in $\D$ such that $g_2(\lk v)\leq 2$.

 Let $u$ be a vertex in $\D$ with $g_2(\lk u)=2$. If $d(\lk u)\geq 8$, then the conclusion can be obtained from Lemmas \ref{g2=5, g2(lku=2), d(u)>=8, no missing edge}, \ref{g2=5, g2(lku=2), two missing edge}, and \ref{g2=5, g2(lku=2), three missing edge}, depending on whether $\lk u$ contains at most one, two, or three missing edges, respectively, present in $\D$. On the other hand, if $d(u)=7$, then the result can be obtained from Lemma \ref{g2=5, g2(lku=2), d(u)=7}.

Let $u$ be a vertex in $\D$ with $g_2(\lk u)=1$. According to Lemma \ref{g2=5, g2(lku=1),d(u)=9}, we can conclude that $d(u)\leq 8$. If $d(u)\geq 7$, we can derive the following conclusions from Lemmas \ref{g2=5, g2(lku=1),d(u)=8} and \ref{g2=5, g2(lku=1),d(u)=7}: $\D$ can be obtained from a triangulated sphere with $g_2\leq 4$ by performing edge expansions together with flipping an edge (possibly zero times) or bistellar $2$-move (possibly zero times). Similarly, if $d(u)=6$, then we can draw the same conclusions based on Lemma \ref{g2=5, g2(lku=1), d(u)=6}.
\end{proof}

\vspace{.15cm}


\noindent {\em Proof of Theorem} \ref{main}. If $\D$ is a prime homology 4-manifold with $g_2(\D)\leq 2$, then $\D$ can be considered as a triangulated sphere, as demonstrated in Propositions \ref{Nevod>3} and \ref{Zhengd>3}. Consequently, the desired result is obtained for the first part of the theorem.

Let $\D$ be a homology $4$-manifold, which may not be prime, with $3 \leq g_2(\D) \leq 5$. Since applying handle addition to a homology 4-manifold increases the $g_2$-value by 15 in the resulting complex, we can express $\D$ as a connected sum of a finite number of prime homology $4$-manifolds, denoted as $\D_1,\dots,\D_n$, each with $g_2$-values less than or equal to 5. For each $\D_i$, we apply Theorems  \ref{prime theorem for g2=3}, \ref{g2=4. prime theorem}, or \ref{g2=5. prime theorem} if $g_2(\D_i)$ is equal to 3, 4, or 5, respectively, and obtain $\D'_i$ with $g_2(\D'_i)<g_2(\D_i)$.

We can then iterate the same procedure for each $\D'_i$ if the value of $g_2$ is greater than 2. These inductive arguments establish the validity of the first part of the theorem. The second part of the theorem is supported by Remark \ref{sharp}.
\hfill $\Box$



\medskip

\noindent {\bf Acknowledgement:} 
The authors would like to thank the anonymous referees for many useful
comments and suggestions. The first author is supported by the Science and Engineering Research Board (CRG/2021/000859). The second author is supported by the Prime Minister's Research Fellows (PMRF/1401215) scheme.

\end{document}